\documentclass[11pt,twoside,reqno]{amsart}
\usepackage[utf8]{inputenc}
\usepackage[T1]{fontenc}
\usepackage{import}
\usepackage{newclude}
\usepackage{amsmath,amssymb,amsthm,mathtools}
\usepackage[toc,page]{appendix}
\usepackage{longtable,booktabs,array}
\usepackage{anysize}
\usepackage{amscd}
\usepackage{stmaryrd}
\usepackage{wasysym}
\usepackage{mathrsfs}
\usepackage[scr=boondox]{mathalfa}
\usepackage{enumitem}
\usepackage{accents}
\usepackage{dsfont}
\usepackage{soul}
\usepackage{color}
\usepackage[all]{xy}
\usepackage{float}
\usepackage{bm}
\usepackage{makecell}
\usepackage{thmtools}      
\usepackage{setspace}
\usepackage{comment}
\usepackage{tikz}
\usepackage{tikz-cd}
\usepackage{mathdots}
\usepackage{mathtools}
\usepackage{graphicx}
\usepackage[myheadings]{fullpage}
\usepackage{bbm}

\usepackage[hypertexnames=false,colorlinks,linkcolor=black,citecolor=black]{hyperref}
\usepackage[capitalize,nameinlink]{cleveref}

\crefname{appendix}{Appendix}{Appendices}
\Crefname{appendix}{Appendix}{Appendices}

\numberwithin{equation}{section}
\newcommand\mtop{.95in}
\newcommand\mbottom{.95in}
\newcommand\mleft{1in}
\newcommand\mright{1in}
\usepackage[top = \mtop, bottom = \mbottom, left = \mleft, right=\mright]{geometry}

\DeclareMathOperator{\Mat}{Mat}

\newtheorem{thm}{Theorem}[section]

\newtheorem{prop}[thm]{Proposition}

\newtheorem{lemma}[thm]{Lemma}

\theoremstyle{definition}
\newtheorem{defi}[thm]{Definition}

\newtheorem{rmk}[thm]{Remark}

\newcommand\reallywidehat[1]{%
\savestack{\tmpbox}{\stretchto{%
  \scaleto{%
    \scalerel*[\widthof{\ensuremath{#1}}]{\kern-.6pt\bigwedge\kern-.6pt}%
    {\rule[-\textheight/2]{1ex}{\textheight}}%WIDTH-LIMITED BIG WEDGE
  }{\textheight}% 
}{0.5ex}}%
\stackon[1pt]{#1}{\tmpbox}%
}
\DeclareSymbolFont{bbold}{U}{bbold}{m}{n}
\DeclareSymbolFontAlphabet{\mathbbold}{bbold}

\makeatletter
\def\@tocline#1#2#3#4#5#6#7{\relax
  \ifnum #1>\c@tocdepth % then omit
  \else
    \par \addpenalty\@secpenalty\addvspace{#2}%
    \begingroup \hyphenpenalty\@M
    \@ifempty{#4}{%
      \@tempdima\csname r@tocindent\number#1\endcsname\relax
    }{%
      \@tempdima#4\relax
    }%
    \parindent\z@ \leftskip#3\relax \advance\leftskip\@tempdima\relax
    \rightskip\@pnumwidth plus4em \parfillskip-\@pnumwidth
    #5\leavevmode\hskip-\@tempdima
      \ifcase #1
       \or\or \hskip 1em \or \hskip 2em \else \hskip 3em \fi%
      #6\nobreak\relax
    \hfill\hbox to\@pnumwidth{\@tocpagenum{#7}}\par% <---- \dotfill -> \hfill
    \nobreak
    \endgroup
  \fi}
\makeatother

\newcommand{\R}{\mathbb{R}}

\newcommand{\Z}{\mathbb{Z}}

\renewcommand{\S}{\mathfrak{S}}

\newcommand{\C}{\mathbb{C}}

\newcommand{\E}{\mathbb{E}}

\newcommand{\mc}{\mathcal}

\newcommand{\bbone}{\mathbbold{1}}
\renewcommand{\l}{\lambda}

\DeclareMathOperator{\Tr}{Tr}

\DeclareMathOperator{\Id}{Id}

\DeclareMathOperator{\Ind}{Ind}
\DeclareMathOperator{\Res}{Res}

\DeclareMathOperator{\Var}{Var}

\DeclareMathOperator{\proj}{proj}

\DeclareMathOperator{\Pois}{Pois}

\DeclareMathOperator{\diag}{diag}

\title{The Cutoff Profile for Random Transpositions on Repeated Cards in the Full Range of Parameters}
\author{Jiahe Shen}

\date{\today}

\begin{document}

\thanks{The author thanks Evita Nestoridi and Dominik Schmid for helpful discussions. The author also thanks Roger Van Peski for reading the draft and providing comments. The author acknowledges support from Ivan Corwin's NSF grant DMS-2246576 and Simons Investigator grant 929852.}

\maketitle

\begin{abstract}
The random transposition shuffle on repeated cards induces a Markov chain on the quotient space of arrangements with multiplicities, and is equivalent to the many-urn mean-field Bernoulli-Laplace model introduced by Scarabotti \cite{scarabotti1997time}. Writing \(n=ml\), where there are \(m\) card types and each type appears \(l\) times, we determine the limiting profile for the total variation distance to stationarity at times
\[
t=\frac n2\left(\log n-\frac12\log l+c\right),
\]
under the assumption \(l=\omega(1)\). Scarabotti  previously established that this process exhibits cutoff at time \(\frac n2(\log n-\frac12\log l)\); our result refines this by identifying the precise asymptotic shape of convergence inside the cutoff window. We show that the limiting profile is asymptotically Gaussian, with different explicit forms in the regimes \(m\) fixed and \(m=\omega(1)\). Together with our previous work on the fixed-\(l\) regime, where the limiting profile is of Poisson type, this yields the cutoff profile for the random transposition shuffle on \(n=ml\) repeated cards for the full range of parameters \(m\) and \(l\).

Our argument has two main steps. First, we combine Scarabotti’s Fourier-analytic framework for the many-urn Bernoulli-Laplace model with the approximation method of Jain-Sawhney \cite{jain2024hitting}. More precisely, we compare the original shuffling measure with an explicitly tractable auxiliary measure directly on the repeated card quotient space, rather than passing through an intermediate comparison on the full symmetric group; this step relies in particular on our new estimates for Kostka numbers. Second, we reduce the limiting-profile problem to quotient fixed-point statistics and analyze them via Hoeffding-type combinatorial central limit theorems.
\end{abstract}



\tableofcontents

\section{Introduction}\label{sec: Intro}

\subsection{Main results}

Card shuffling has long served as a central source of examples and problems in the study of mixing times for Markov chains. Among the many models that have been studied, the random transposition shuffle occupies a particularly distinguished position. In their seminal work, Diaconis-Shahshahani \cite{diaconis1981generating} proved that the random transposition walk on \(S_n\) exhibits cutoff at time \(\frac12 n\log n\), thereby providing one of the earliest and most influential examples of the cutoff phenomenon for a natural Markov chain. Much more recently, Teyssier \cite{teyssier2020limit} refined this result by determining the limiting profile inside the cutoff window. Thus, for the random transposition shuffle on the symmetric group, both the cutoff location and the precise asymptotic shape of convergence in total variation are now understood.

Since then, a rich literature has developed around related shuffling models on permutations. These include, among many others, riffle shuffling \cite{bayer1992riffle}, star transpositions \cite{nestoridi2024comparing}, random \(k\)-cycles \cite{berestycki2011mixing,hough2016random, nestoridi2022limit}, and adjacent transpositions \cite{lacoin2016adjacent}, each of which exhibits its own characteristic mixing behavior and often requires rather different techniques. In the present paper, we study the random transposition shuffle in the setting of repeated cards, where the state space is no longer the full symmetric group itself but the space of repeated card arrangements. Our goal is to determine the cutoff profile in this setting for the full range of parameters.

We begin by introducing the quotient-space description of repeated card arrangements. For all $n\ge 1$, let $\S_n$ be the symmetric group over $n$ symbols, and $U_{\S_n}$ be the uniform distribution over $\S_n$. 

\begin{defi}\label{defi: quotient of space}
Suppose $n=ml$ where $m,l\ge 1$. Then we abbreviate $\sim_l\backslash\S_n$ for the set of left cosets
$$(\underbrace{\S_l\times\cdots\times \S_l}_{m\text{ times}})\backslash\S_n:=\{(\underbrace{\S_l\times\cdots\times \S_l}_{m\text{ times}})\circ\sigma:\sigma\in\S_n\},$$
where 
$$\underbrace{\S_l\times\cdots\times \S_l}_{m\text{ times}}:=\{\sigma\in\S_n:\sigma(i)\equiv i\pmod m\quad\forall1\le i\le n\}$$
is a subgroup of $\S_n$. 
Denote by $\sim_l\backslash U_{\S_n}$ the uniform distribution over $\sim_l\backslash\S_n$.
\end{defi}

In other words, $\sim_l\backslash\S_n$ is exactly the state space of repeated card arrangements with $n=ml$ cards, where there are $m$ card types and each type appears $l$ times; two permutations are identified whenever they differ only by permuting the $l$ copies of each type.

Let us now recall the definition of the random transposition shuffle. As is standard in the literature on random transpositions (see, for instance, Diaconis-Shahshahani \cite{diaconis1981generating} and Teyssier \cite{teyssier2020limit}), we work with the lazy version of the chain, meaning that the identity permutation is chosen with nonzero probability.

\begin{defi}\label{defi: Pnk and its convolution}
Let $P_n$ be the random permutation given by
\begin{equation}\label{eq: measure of k-cycle}
P_n:=\begin{cases}
(i j) & \text{with probability }\frac{2}{n^2}\\
\Id & \text{with probability }\frac{1}{n}\end{cases}.
\end{equation}
Moreover, for any $t\ge 1$, let $P_{n}^{*t}$ be the $t$-fold convolution of $P_n$, viewed $P_{n}^{*t}$ as a random variable in $\S_n$. Given $n=ml$ with $l\ge 2$, the random variable $\sim_l\backslash P_{n}^{*t}\in\sim_l\backslash\S_n$ is induced from \Cref{defi: quotient of space} in the obvious way.
\end{defi}

We now state our main theorem, which determines the cutoff profile for the random transposition shuffle on repeated cards in the regime \(l=\omega(1)\). 

\begin{thm}[\Cref{thm: restate of limiting profile} in the text]\label{thm: limiting profile_main thm}
Let $n=ml$, such that $m,l\ge 2$, and $l=\omega(1)$. Suppose 
$$t=\frac n2\left(\log n-\frac 12\log l+c\right),$$ 
where $c\in \R$ is an absolute constant. Then, we have 
$$d_{TV}(\sim_l\backslash P_{n}^{*t},\sim_l\backslash U_{\S_n})=\begin{cases}
2\Phi\left(\frac{\exp(-c)}{2}\right)-1+o(1) & m=\omega(1)\\
2\Phi\left(\sqrt{1-\frac 1m}\frac{\exp(-c)}{2}\right)-1+o(1) & m\ge 2 \text{ fixed}
\end{cases}.$$
Here, $\Phi(x):=\frac{1}{\sqrt{2\pi}}\int_{-\infty}^x\exp(-y^2/2)dy$ denotes the cumulative distribution function of the normal distribution $\mathcal{N}(0,1)$.
\end{thm}

In particular, \Cref{thm: limiting profile_main thm} shows that the random transposition shuffle on repeated cards exhibits cutoff at time
\[
\frac n2\left(\log n-\frac12\log l\right)
\]
with window of order \(n\). Together with the author’s previous work \cite{shen2026k} on the fixed-\(l\) regime \footnote{Strictly speaking, the author’s previous paper \cite{shen2026k} treated the non-lazy version of the shuffle, in which the identity permutation has probability zero at each step. However, Jain and Sawhney’s approximation theorem \cite[Theorem 1.3]{jain2024hitting} is stated for the standard lazy random transposition walk, where the identity is chosen with nonzero probability. Combining that result directly with the counting argument in \cite[Section 3]{shen2026k}, one can derive the same limiting profile for the lazy repeated card shuffle by essentially the same method.}, this yields the cutoff profile for the random transposition shuffle on \(n=ml\) repeated cards over the full range of parameters \(m\) and \(l\). In the fixed-\(l\) regime, the limiting total variation profile is of Poisson type, whereas \Cref{thm: limiting profile_main thm} shows that once \(l=\omega(1)\), the profile becomes Gaussian. Thus, taken together, these results give a complete description of the limiting profile for Scarabotti’s many-urn mean-field Bernoulli-Laplace model, which is equivalent to the repeated card random transposition shuffle after passing to the natural quotient that identifies repeated copies of the same card type.

From the viewpoint of the Bernoulli-Laplace diffusion model with many urns, this identifies the limiting profile at the cutoff scale for a Markov process whose cutoff location was already established by Scarabotti \cite{scarabotti1997time}. In this sense, the present theorem complements Scarabotti’s result by passing from the cutoff phenomenon to the precise asymptotic shape of convergence inside the cutoff window. Moreover, for the open problem of Goenka-Hermon-Schmid \cite[Question 1.9]{goenka2025cutoff} concerning the limiting profile of the balanced chain, the present theorem resolves the mean-field many-urn special case, namely Scarabotti's Bernoulli-Laplace model on the complete graph, or equivalently the repeated card random transposition shuffle on the quotient space. To the best of our knowledge, this is the first result that determines the cutoff profile for this repeated card random transposition shuffle (or equivalently, this many-urn Bernoulli-Laplace process) over the full range of parameters.

Our result also parallels a broader line of work on mixing and cutoff for random walks on homogeneous quotient spaces. In particular, Nestoridi and Olesker-Taylor \cite{nestoridi2022limit} developed a general framework for limit profiles of reversible Markov chains, and later studied limit profiles for projections of random walks on groups in a systematic way. Repeated card shuffling models have also been studied previously in the context of riffle shuffles: for instance, Conger-Viswanath \cite{conger2006riffle} analyzed riffle shuffles of decks with repeated cards, and further investigated related questions motivated by card games such as blackjack and bridge; see also the work of Assaf-Diaconis-Soundararajan \cite{assaf2009riffle} on riffle shuffles of a deck with repeated cards. Broadly speaking, this line of work relies on explicit combinatorial descriptions of the induced transition probabilities and on the study of descent-type statistics for multiset permutations.

Our approach in this paper is, however, rather different from the directions described above. To study the repeated card model, we work within the Fourier-analytic framework underlying Scarabotti’s analysis of the many-urn Bernoulli-Laplace process. At the same time, we import into this setting the approximation method of Jain-Sawhney \cite{jain2024hitting}, which was originally developed for the classical random transposition shuffle without repeated copies. Roughly speaking, their method replaces the original shuffling measure by an explicitly tractable auxiliary measure that remains asymptotically close to the original walk in the relevant time regime, while being much easier to analyze, as given by the following. 

\begin{defi}\label{defi: construction of nu}
Suppose $t\in\Z$. Set $t':=t-\lfloor\frac{n\log n}{2}\rfloor$, and $\gamma_t:=\exp(-2t'/n)$. Let $\nu_n^t$ denote the measure on $\S_n$ defined by the following sampling process: first, sample $M_t\in\{0,1,\ldots,n\}$ according to the distribution 
$$\mathbf{P}[M_t=k]=\mathbf{P}(\Pois(\gamma_t)=k)/\mathbf{P}(\Pois(\gamma_t)\le n),$$ 
then sample a uniformly random subset $S_t$ of size $M_t$ in $[n]$. Finally, sample a uniformly random element of $\S_{[n]\backslash S_t}$ and view it as an element of $\S_n$ by fixing all of the elements in $S_t$.
\end{defi}

In the setting of \Cref{thm: limiting profile_main thm}, we will always take
\[
t=\frac n2\left(\log n-\frac 12\log l+c\right),
\]
so that
\[
\gamma_t=\exp(-c)\sqrt l.
\]
The role of \Cref{defi: construction of nu} is to introduce an explicitly tractable auxiliary measure at exactly this timescale. The next theorem shows that, on the repeated card quotient space, this measure approximates the original random transposition shuffle asymptotically well, and thus provides the main input toward the proof of \Cref{thm: limiting profile_main thm}.

\begin{thm}[\Cref{thm: restate of pnkt and nu} in the text]\label{thm: pnkt and nu}
Let $n=ml$, such that $m,l\ge 2$, and $n=\omega(1)$. Suppose 
$$t=\frac n2\left(\log n-\frac 12\log l+c\right),$$ 
where $c\in \R$ is an absolute constant. Let $\nu_n^t$ be the probability measure over $\S_n$ defined the same way as in \Cref{defi: construction of nu}, and let $\sim_l\backslash \nu_n^t$ be the probability measure over $\sim_l\backslash\S_n$ induced by $\nu_n^t$. Then, we have 
$$d_{TV}(\sim_l\backslash P_n^{*t},\sim_l\backslash\nu_n^t)\le n^{-1+o(1)}.$$
\end{thm}

In order to prove \Cref{thm: pnkt and nu}, we still adopt a Fourier-analytic approach, as in the work of Jain-Sawhney. More precisely, our goal is to compare the original shuffling measure with the explicitly tractable auxiliary measure \(\nu_n^t\) through their Fourier transforms. However, one should note that if one works on the full symmetric group \(\mathfrak S_n\), namely without passing to the repeated card quotient space, then the approximation between \(P_n^{\ast t}\) and \(\nu_n^t\) is accurate only when \(t\) lies very close to \(\frac{n\log n}{2}\), which is the mixing-time scale of the classical random transposition walk. In the repeated card setting considered here, the relevant time scale is instead
\[
t=\frac n2\left(\log n-\frac12\log l+c\right),
\]
which may be substantially smaller.

For this reason, we must carry out the comparison directly on the quotient space, rather than first comparing \(P_n^{\ast t}\) and \(\nu_n^t\) on \(\S_n\). From this point of view, our argument remains within the same Fourier-analytic framework as Scarabotti’s study of the many-urn Bernoulli-Laplace model. In particular, after passing to the quotient by \(H=(\S_l)^m\), the relevant \(L^2\)-comparison may be expressed as a spectral sum over irreducible representations of \(\S_n\); see \Cref{prop: quotient L2 formula}. The contribution of a partition \(\lambda\vdash n\), viewed as indexing the corresponding irreducible representation \(V_\lambda\) of \(\S_n\), is weighted by
\[
d_\lambda\,\dim(V_\lambda^H),
\]
where \(d_\lambda=\dim(V_\lambda)\), and \(V_\lambda^H\) denotes the \(H\)-fixed subspace of \(V_\lambda\), namely the subspace consisting of vectors invariant under the action of every element of \(H\). A key observation is that, by Frobenius reciprocity, the multiplicity \(\dim(V_\lambda^H)\) can be identified with the Kostka number \(K_{\lambda,(l^m)}\). Thus the quotient-level comparison naturally leads to spectral sums weighted by \(d_\lambda K_{\lambda,(l^m)}\).

This reduction brings in a new combinatorial input. In the classical random transposition shuffle on \(\S_n\), one only needs to control the dimensions \(d_\lambda=K_{\lambda,(1^n)}\), as in the work of Diaconis-Shahshahani \cite{diaconis1981generating}. In the repeated card setting, however, one must instead bound the more general Kostka numbers \(K_{\lambda,(l^m)}\). We obtain new estimates for these Kostka numbers, and combine them with the relevant character bounds to control the quotient-level spectral sum, which in turn yields \Cref{thm: pnkt and nu}.

We now explain how to deduce \Cref{thm: limiting profile_main thm} from \Cref{thm: pnkt and nu}. By \Cref{thm: pnkt and nu}, it suffices to replace the original random transposition shuffle by the auxiliary measure \(\nu_n^t\), at the cost of an asymptotically negligible error in total variation distance on the quotient space. Therefore, the proof of \Cref{thm: limiting profile_main thm} is reduced to analyzing the total variation distance between \(\sim_l\backslash \nu_n^t\) and the uniform measure on \(\sim_l\backslash \S_n\). The advantage of passing to \(\nu_n^t\) is that it admits an explicit probabilistic description in terms of a random fixed set together with a uniformly random permutation on its complement. As a result, the problem can be further reduced to understanding the corresponding quotient fixed-point statistics. The advantage of this reformulation is that these statistics are amenable to probabilistic analysis: in \Cref{sec: From tractable measure to total variation distance}, we use Hoeffding-type combinatorial central limit theorems to derive the asymptotic Gaussian behavior of the quotient fixed-point statistics, together with several auxiliary tail bounds established by probabilistic methods. This ultimately leads to the proof of \Cref{thm: limiting profile_main thm}.

It is worth mentioning that the quotient-space method developed above is not restricted to the balanced repeated-card setting in which every card type appears with the same multiplicity. For example, the classical Bernoulli-Laplace urn model studied in Olesker-Taylor and Schmid \cite{olesker2024limit} corresponds, in the present language, to the repeated-card quotient with two card types, one appearing \(n-k\) times and the other appearing \(k\) times, where \(k\leq n/2\); equivalently, the relevant subgroup is
\[
\S_{n-k}\times \S_k\leq \S_n .
\]
In this case the quotient-level Fourier expansion is governed by the Kostka number
\[
K_{\lambda,(n-k,k)}.
\]
This Kostka number has a particularly simple form: it is equal to \(1\) when \(\lambda=(n)\), or when \(\lambda\) has two rows and its first row has length at least \(n-k\), and it is equal to \(0\) otherwise. Consequently, the spectral sums arising from our quotient-space comparison become completely explicit in this two-type setting. In particular, combining this explicit Kostka-number description with the approximation method of Jain-Sawhney should give an alternative representation-theoretic proof of \cite[Theorem A]{olesker2024limit}, which determines the limiting profile for the Bernoulli-Laplace urn. This is related to a point made in the first version of \cite{olesker2024limit}: the authors noted that, although a representation-theoretic proof of the limiting profile would be desirable, they were not able to obtain one at that time. From the perspective of the present paper, the missing ingredient is precisely the approximation method of Jain-Sawhney, which allows one to replace the original random transposition walk by an explicitly tractable auxiliary measure before carrying out the quotient-level spectral analysis.

\subsection{What's next?}

A natural direction for future work is to extend the present quotient-space approach to the random \(k\)-cycle shuffle on repeated cards. At the representation-theoretic level, the relevant background should come from the existing theory of random \(k\)-cycles on \(\S_n\), notably the character estimates from Hough \cite{hough2016random} and the later limit-profile framework of Nestoridi and Olesker-Taylor \cite{nestoridi2022limit}. In addition, the recent sharp character bounds of Olesker-Taylor, Teyssier, and Thévenin \cite{olesker2025sharp} suggest that one may hope to treat substantially larger values of \(k\) than were previously accessible. This raises the question of whether the quotient-level methods developed in the present paper can be combined with these representation-theoretic inputs to study cutoff and cutoff profiles for \(k\)-cycle shuffling with repeated cards. At least heuristically, once one is able to establish an approximation by a suitably tractable auxiliary measure, in the spirit of \Cref{thm: pnkt and nu}, the remaining task should be to determine the limiting profile by analyzing the corresponding quotient fixed-point statistics, much as in \Cref{sec: From tractable measure to total variation distance}. The main difficulty in extending our argument to random \(k\)-cycles is therefore expected to lie in the approximation step. Compared with the random transposition case, this would require a more delicate analysis of the relevant representations and, in particular, character estimates that are sufficiently fine for the quotient-level spectral bounds needed here.

Another natural direction is to study other shuffling mechanisms on repeated cards. For instance, one may consider the star transposition shuffle, for which the comparison philosophy developed by Nestoridi \cite{nestoridi2024comparing} in her work on comparing limit profiles of reversible Markov chains seems especially relevant. It is therefore natural to ask whether one can combine the comparison method from that line of work with the quotient-level analysis developed in the present paper in order to determine the limiting profile for star transposition shuffling with repeated cards. 

A further direction is suggested by the \(k\)-particle interchange process on the complete graph, whose state space can be viewed as the homogeneous space \(H\backslash G\) with \(G=\S_n\) and \(H=\S_k\times \S_1\times\cdots\times \S_1\). Equivalently, this is the quotient space obtained by identifying permutations that differ only by relabeling the \(k\) marked particles among themselves. In \cite{nestoridi2022limit_projectionsofrandomwalk}, Nestoridi and Olesker-Taylor proved a Poisson limit profile when \(k/n\to\alpha\in(0,1]\), and moreover formulated a conjecture predicting that in the sparse regime \(1\ll k\ll n\), the total variation profile should instead be Gumbel at time \(\frac n2(\log k+c)\). From the perspective of the present paper, it seems reasonable to investigate whether one can approach this conjecture through suitable estimates for the relevant Kostka numbers.

\subsection{Notation}

Throughout the paper, \(d_{TV}(\cdot,\cdot)\) denotes total variation distance, \(\Pois(\cdot)\) denotes the Poisson distribution, and \(\Phi(\cdot)\) denotes the cumulative distribution function of the standard normal distribution \(\mathcal N(0,1)\).

We write \([n]:=\{1,2,\ldots,n\}\), and \(2^{[n]}\) for the collection of all subsets of \([n]\). When \(n=ml\), for each \(1\le i\le m\) we write
\[
T_{i,l}:=\{i,i+m,\ldots,i+(l-1)m\},
\]
so that \([n]=\bigsqcup_{i=1}^m T_{i,l}\). Thus \(T_{i,l}\) is the set of positions occupied by the \(i\)-th card type. We use \(|\cdot|\) to denote either the cardinality of a finite set or the absolute value of a complex number.

We write \(f=O(g)\) if \(|f|\le C|g|\) for some absolute constant \(C>0\), \(f=o(g)\) if \(f/g\to0\), and \(f=\omega(g)\) if \(f/g\to+\infty\). All asymptotic notation is understood in absolute value. Moreover, the implicit constants are taken along the given sequence of positive integer inputs; in particular, they are not intended uniformly over all such inputs, but only along the sequence under consideration.

We write \(\xrightarrow{d}\) for weak convergence, \(\mathbf P[\cdot\mid\cdot]\) for conditional probability, and \(\mathbf E[\cdot\mid\cdot]\) for conditional expectation. We also continue to use the notation \(\sim_l\backslash\) from \Cref{defi: quotient of space} to denote the repeated card quotient space and the corresponding induced measures.

\subsection{Outline of the paper}

In \Cref{sec: Approximating the shuffling with an explicitly tractable measure}, we first introduce the nonabelian Fourier transform and other necessary representation-theoretic background, and then use these tools to prove \Cref{thm: pnkt and nu}.

In \Cref{sec: From tractable measure to total variation distance}, based on the approximation provided by \Cref{thm: pnkt and nu}, we prove \Cref{thm: limiting profile_main thm}. We treat separately the two regimes where \(m\) is fixed and where \(m=\omega(1)\), and work out the details in each case.

\section{Approximating the shuffling with an explicitly tractable measure}\label{sec: Approximating the shuffling with an explicitly tractable measure}

The purpose of this section is to prove the following theorem, which gives a detailed formulation of \Cref{thm: pnkt and nu}.

\begin{thm}\label{thm: restate of pnkt and nu}
Let $c\in\R$ be an absolute constant. Let $(n_N)_{N\ge 1},(m_N)_{N\ge 1}$ and $(l_N)_{N\ge 1}$ be sequences of positive integers, such that 
\begin{enumerate}
\item $m_N,l_N\ge 2$ for all $N\ge 1$.
\item $n_N=m_Nl_N$ for all $N\ge 1$, and $\lim_{N\rightarrow\infty}n_N=\infty$.
\end{enumerate}
For all $N\ge 1$, let $t_N=\frac{n_N}{2}(\log n_N-\frac 12\log l_N+c)$. Let $P_{n_N}^{*t_N}$ be defined as in \Cref{defi: Pnk and its convolution}, and $\nu_{n_N}^{t_N}$ be defined as in \Cref{defi: construction of nu}. Let $\sim_{l_N}\backslash P_{n_N}^{*t_N},\sim_{l_N}\backslash \nu_{n_N}^{t_N}$ be their induced probability measures over $\sim_{l_N}\backslash\S_{n_N}$. Then, we have
$$\limsup_{N\rightarrow\infty}\log n_N(d_{TV}(\sim_{l_N}\backslash P_{n_N}^{*t_N},\sim_{l_N}\backslash\nu_{n_N}^{t_N}))\le -1.$$
\end{thm}

For the rest of the section, for notational ease, we drop off the subscripts and simply write $n,m,l$ for $n_N,m_N,l_N$. Now, let us recall necessary background from representation theory.

\subsection{Nonabelian Fourier transform and related notations}

Given a finite group $G$, we let $\widehat G$ denote the set of irreducible representations. We define convolution $f_1,f_2:G\rightarrow\C$ to be
$$(f_1*f_2)(z):=\sum_{x\in G}f_1(x)f_2(x^{-1}z).$$
For $\rho\in\widehat G$, let $V_\rho$ be the representation space of $\rho$, and let $d_\rho=\dim V_\rho$ denote the dimension of $V_\rho$. The nonabelian Fourier transform for a function $f:G\rightarrow\C$ is the map $\widehat f$ given by
$$\widehat f(\rho):=\sum_{g\in G}f(g)\rho(g)\in\Mat_{d_\rho}(\C),\quad\forall \rho\in\widehat G.$$
In this way, the mapping $f\mapsto(\widehat f(\rho))_{\rho\in\hat G}$ gives an isomorphism between $\C$-algebras $\C[G]$ and $\prod_{\rho\in\hat G}\Mat_{d_\rho}(\C)$. 

We benefit from the above Fourier transform in several ways. First, it transfers convolution into multiplication, i.e., for all $f_1,f_2:G\rightarrow\C$ and $\rho\in\widehat G$, we have
\begin{equation}\label{eq: convolution to multiplication}
\widehat{f_1*f_2}(\rho)=\widehat f_1(\rho)\cdot\widehat f_2(\rho).
\end{equation}
Second, ensured by Parseval's theorem, it preserves the $L^2$ norm on both sides. Explicitly, let $\{e_g\}_{g\in G}$ be the natural $\C$-basis of $\C[G]$. Then, for $f=\sum_{g\in G}a_g e_g\in\C[G]$, we have
$$\sum_{g\in G}|a_g|^2=\frac{1}{|G|}\sum_{\rho\in\hat G}d_\rho\Tr\left(\hat f(\rho)\hat f(\rho)^\dagger\right).$$
Here, the dagger superscript refers to the conjugate transpose matrix.

Now, suppose $H$ is a subgroup of $G$. For all $\rho\in \hat G$, let
$$
V_\rho^H:=\{v\in V_\rho:\rho(h)v=v\ \text{for all }h\in H\}
$$
be the $H$-fixed subspace, and
$$(V_\rho^H)^\perp:=\{v\in V_\rho:\sum_{h\in H}\rho(h)v=0\},$$
which gives natural decomposition of $H$-modules $V_\rho=V_\rho^H\oplus(V_\rho^H)^\perp$. We denote by $\proj_{V_\rho^H}\in\Mat_{d_\rho}(\C)$ the projection map onto $V_\rho^H$. 

\begin{prop}\label{prop: quotient L2 formula}
Let \(G\) be a finite group and \(H\leq G\) be a subgroup. Let
\(\mu_1,\mu_2:G\to\C\) be two class functions which are probability measures on \(G\).
For each irreducible representation \(\rho\) of \(G\), let
\(a_\rho^{(1)},a_\rho^{(2)}\in\C\) be the scalars such that
\begin{equation}
\sum_{g\in G}\mu_i(g)\rho(g)=a_\rho^{(i)}I_{d_\rho},
\qquad i=1,2.
\end{equation}
Denote by \(\sim_H\backslash\mu_1,\sim_H\backslash\mu_2\) the probability measures on
the left coset space \(H\backslash G\) induced by \(\mu_1,\mu_2\), respectively. Then, we have
\begin{equation}\label{eq: L2 quotient formula}
D_{L^2}(\sim_H\backslash\mu_1,\sim_H\backslash\mu_2)^2
=
\frac{|H|}{|G|}
\sum_{\rho}
d_\rho\,\dim(V_\rho^H)\,
\left|a_\rho^{(1)}-a_\rho^{(2)}\right|^2.
\end{equation}
Here, \(D_{L^2}\) denotes the \(L^2\)-distance between measures on the quotient space
\(H\backslash G\), while on the right-hand side, we sum over all nontrivial irreducible representations $\rho$ of $G$.
\end{prop}

\begin{proof}%[Proof of \Cref{prop: quotient L2 formula}]
Let $u:=\mu_1-\mu_2$. Since \(\mu_1,\mu_2\) are class functions, so is \(u\). Moreover, subtracting the two scalar
identities gives
\begin{equation}\label{eq: scalar difference quotient}
\sum_{g\in G}u(g)\rho(g)
=
\bigl(a_\rho^{(1)}-a_\rho^{(2)}\bigr)I_{d_\rho}
\end{equation}
for every irreducible representation \(\rho\) of \(G\).

Since \(\mu_1,\mu_2\) are probability measures on \(G\), we have
\[
\sum_{g\in G}u(g)=0.
\]
In particular, if \(\rho\) is the trivial representation, then
\[
a_\rho^{(1)}-a_\rho^{(2)}=0.
\]

Now define
\[
F:=e_H*u,
\qquad\text{where}\qquad
e_H:=\sum_{h\in H}e_h.
\]
Then for every \(x\in G\),
\[
F(x)=\sum_{h\in H}u(h^{-1}x).
\]
Since \(h\mapsto h^{-1}\) is a bijection on \(H\), this becomes
\[
F(x)=\sum_{h\in H}\mu_1(hx)-\sum_{h\in H}\mu_2(hx).
\]
Thus \(F\) is constant on each left coset \(Hx\in H\backslash G\), and therefore
\begin{equation}\label{eq: quotient norm compare}
\|F\|_{L^2(G)}^2
=
|H|\,D_{L^2}(\sim_H\backslash\mu_1,\sim_H\backslash\mu_2)^2.
\end{equation}

We now compute the Fourier transform of \(F\). For each irreducible representation \(\rho\),
using \eqref{eq: scalar difference quotient}, we obtain
\[
\sum_{x\in G}F(x)\rho(x)
=
\left(\sum_{h\in H}\rho(h)\right)
\left(\sum_{g\in G}u(g)\rho(g)\right)
=
\left(\sum_{h\in H}\rho(h)\right)
\bigl(a_\rho^{(1)}-a_\rho^{(2)}\bigr)I_{d_\rho}.
\]
By the decomposition
\[
V_\rho=V_\rho^H\oplus (V_\rho^H)^\perp,
\]
the operator \(\sum_{h\in H}\rho(h)\) acts by \(|H|\) on \(V_\rho^H\) and by \(0\) on
\((V_\rho^H)^\perp\). Hence
\[
\sum_{h\in H}\rho(h)=|H|\,\mathrm{proj}_{V_\rho^H},
\]
and so
\[
\sum_{x\in G}F(x)\rho(x)
=
|H|\bigl(a_\rho^{(1)}-a_\rho^{(2)}\bigr)\mathrm{proj}_{V_\rho^H}.
\]
Taking Hilbert--Schmidt norms, we get
\[
\left\|
\sum_{x\in G}F(x)\rho(x)
\right\|_{HS}^2
=
|H|^2\left|a_\rho^{(1)}-a_\rho^{(2)}\right|^2
\|\mathrm{proj}_{V_\rho^H}\|_{HS}^2.
\]
Since \(\mathrm{proj}_{V_\rho^H}\) is an orthogonal projection,
\[
\|\mathrm{proj}_{V_\rho^H}\|_{HS}^2=\dim(V_\rho^H).
\]
Therefore
\[
\left\|
\sum_{x\in G}F(x)\rho(x)
\right\|_{HS}^2
=
|H|^2\dim(V_\rho^H)\left|a_\rho^{(1)}-a_\rho^{(2)}\right|^2.
\]

Applying Plancherel's formula on \(G\), we obtain
\[
\|F\|_{L^2(G)}^2
=
\frac1{|G|}
\sum_{\rho}
d_\rho
\left\|
\sum_{x\in G}F(x)\rho(x)
\right\|_{HS}^2
=
\frac{|H|^2}{|G|}
\sum_{\rho}
d_\rho\,\dim(V_\rho^H)\,
\left|a_\rho^{(1)}-a_\rho^{(2)}\right|^2.
\]
Since the trivial representation contributes \(0\), the sum may be restricted to all nontrivial
irreducible representations of \(G\). Comparing this with \eqref{eq: quotient norm compare},
we conclude that
\[
D_{L^2}(\sim_H\backslash\mu_1,\sim_H\backslash\mu_2)^2
=
\frac{|H|}{|G|}
\sum_{\rho}
d_\rho\,\dim(V_\rho^H)\,
\left|a_\rho^{(1)}-a_\rho^{(2)}\right|^2,
\]
where the sum ranges over all nontrivial irreducible representations \(\rho\) of \(G\).
This proves the proposition.
\end{proof}

In this paper, we always take $G=\mathfrak{S}_n$ and $H=(\S_l)^m$. Following the standard construction of Specht modules, the irreducible representations of $\mathfrak{S}_n$ are indexed by partitions of $n$. In particular, the trivial irreducible representation corresponds to the partition $(n)$. Due to Frobenius reciprocity, the dimension of $V_\lambda^{(\S_l)^m}$ is given by
$$\dim V_\lambda^{(\S_l)^m}=\left\langle\bbone_{(\S_l)^m},\Res_{(\S_l)^m}^{\S_n}\chi_\lambda\right\rangle_{(\S_l)^m}=\left\langle\Ind_{(\S_l)^m}^{\S_n}\bbone_{(\S_l)^m},\chi_\lambda\right\rangle_{\S_n}=K_{\lambda,(l^m)}.$$
Here, $K_{\lambda,(l^m)}$ denotes the Kostka number. From now on, we use $\l\vdash n$ to denote that $\l$ is a partition of $n$. In particular, we have $d_\lambda=d_{\lambda'}=K_{\lambda,(1^n)}$ when $\l$ has size $n$. Moreover, the fact $\sum_{\l\vdash n}d_\l^2=n!$ implies the bound $d_\l\le \sqrt {n!}$ for all $\l\vdash n$.

Given a partition $\l=(\l_1,\ldots,\l_j)$ with $\l_1\ge\cdots\ge\l_j$, we let $\l'=(\l_1',\l_2',\ldots)$ denote the conjugate partition, $\l^*=(\l_2,\ldots,\l_j)$ be the partition obtained from $\l$ by deleting its first row, and
\[
\lambda^\bullet:=\bigl((\lambda')^*\bigr)'\,,
\]
which is obtained from $\lambda$ by deleting its first column.

Let 
\begin{equation}\label{eq: set of long first row partitions}
\mathcal L \ :=\ \Bigl\{\lambda\vdash n:\ \lambda_1\ge n-\log n\Bigr\}
\end{equation}
be the set of partitions with long first row. For the rest of this section, we  usually use the letter $r$ for $n-\l_1$.

Probability measures on $\mathfrak{S}_n$ shall be viewed as real-valued functions. In our applications, we will furthermore restrict attention to functions $f$ which are class functions, i.e., functions that are constants on conjugacy classes. 
\begin{lemma}[Schur's lemma]
Suppose $f:\mathfrak{S}_n\rightarrow \C$ is a class function. Then, we have
$$\widehat f(\lambda)=\frac{\sum_{g\in G}f(g)\chi_\lambda(g)}{d_\lambda}\Id_{d_\lambda},$$
where $\chi_\lambda$ is the character corresponding to $\lambda\vdash n$.
\end{lemma}

In particular, let $P_n$ be the probability measure over $\mathfrak{S}_n$ as in \eqref{eq: measure of k-cycle}. Then, its Fourier transform is given by
\begin{equation}\label{eq: fourier transform of Pn}
\widehat P_{n}(\lambda)=\left(\frac 1n+\frac{n-1}{n}\frac{\chi_\lambda(\tau)}{d_\lambda}\right)\Id_{d_\lambda},\quad \lambda\vdash n.
\end{equation}
where $\tau$ is an arbitrary transposition. In particular, by the hook length formula, we have $d_\l=d_{\l'}$.

% $$\widehat{U_{\mathfrak{S}_n}}(\lambda)=\begin{cases}
% \Id_{d_\l} & \lambda=(n)\\
% 0 & \text{else}
% \end{cases},$$
% $$\widehat{U_{\mathfrak{A}_n}}(\lambda)=\begin{cases}
% \Id_{d_\l} & \lambda=(1^n),(n)\\
% 0 & \text{else}
% \end{cases},$$
% $$\widehat{U_{\mathfrak{A}_n^c}}(\lambda)=\begin{cases}
% \Id_{d_\l} & \lambda=(1^n)\\
% -\Id_{d_\l} & \lambda=(n)\\
% 0 & \text{else}
% \end{cases}.$$

\subsection{Character estimates for $P_n^{*t}$}

Following \eqref{eq: convolution to multiplication} and \eqref{eq: fourier transform of Pn}, it is clear that
$$\widehat {P_n^{*t}}(\l)=\left(\frac 1n+\frac{n-1}{n}\frac{\chi_\lambda(\tau)}{d_\lambda}\right)^t\Id_{d_\l},\quad\l\vdash n.$$

\begin{lemma}[{\cite[Lemma 7]{diaconis1981generating}}]\label{lem: explicit expression of character}
Let $\lambda\vdash n$ and $\tau\in\S_n$ be a transposition. Then, we have
$$\chi_\lambda(\tau)=-\chi_{\lambda'}(\tau)=d_\lambda\binom n2^{-1}\left(\sum_{i=1}^n\binom {\lambda_i}{2}-\binom {\lambda_i'}{2}\right).$$
\end{lemma}

\begin{lemma}
For all $\lambda=(\lambda_1,\lambda_2,\ldots)\vdash n$, we have
\begin{equation}\label{eq: short first row}
-\frac{\lambda_1'}{n}\le\frac 1n+\frac{n-1}{n}\frac{\chi_\l(\tau)}{d_\l}\le\frac{\lambda_1}{n},
\end{equation}
\begin{equation}\label{eq: middle first row}
-1+2\frac{\lambda_1}{n}\le\frac 1n+\frac{n-1}{n}\frac{\chi_\l(\tau)}{d_\l}\le1-2\frac{\lambda_1}{n}+2\left(\frac{\lambda_1}{n}\right)^2.
\end{equation}
In particular, when $\lambda_1\ge n-n^{o(1)}$, set $r:=n-\l_1$. Then, we have
\begin{equation}\label{eq: long first row}
\frac 1n+\frac{n-1}{n}\frac{\chi_\l(\tau)}{d_\l}=1-2\frac{r}{n}+O(n^{-2+o(1)}).
\end{equation}
\end{lemma}

\begin{proof}
We first prove \eqref{eq: short first row}. In fact, following the expression in \Cref{lem: explicit expression of character}, we obtain the upper bound
$$\frac 1n+\frac{n-1}{n}\frac{\chi_\l(\tau)}{d_\l}\le\frac 1n+\frac{\sum_{i=1}^n(\lambda_i^2-\lambda_i)}{n^2}=\frac{\sum_{i=1}^n\lambda_i^2}{n^2}\le\frac{\sum_{i=1}^n\lambda_i\lambda_1}{n^2}=\frac{\lambda_1}{n}.$$
The lower bound of \eqref{eq: short first row} follows similarly.

We next prove \eqref{eq: middle first row}. On the one hand, we have the upper bound
$$\frac 1n+\frac{n-1}{n}\frac{\chi_\l(\tau)}{d_\l}\le\frac{\sum_{i=1}^n\lambda_i^2}{n^2}\le\frac{\lambda_1^2+(\sum_{i=2}^n\lambda_i)^2}{n^2}=1-2\frac{\lambda_1}{n}+2\left(\frac{\lambda_1}{n}\right)^2.$$
On the other hand, we have $\sum(\lambda_i'-1)=n-\lambda_1$, where the sum ranges over all positive parts of $\lambda'$. This leads to the lower bound
\begin{align}
\begin{split}
\frac 1n+\frac{n-1}{n}\frac{\chi_\l(\tau)}{d_\l}&\ge\frac 1n+\frac{\lambda_1^2-\lambda_1-\sum_{i=1}^n\lambda_i'(\lambda_i'-1)}{n^2}\\
&\ge\frac 1n+\frac{\lambda_1^2-\lambda_1-(n-\lambda_1)-(n-\lambda_1)^2}{n^2}\\
&=-1+2\frac{\l_1}{n}.
\end{split}
\end{align}
Finally, one can deduce \eqref{eq: long first row} directly from the upper and lower bound given in \eqref{eq: middle first row}.
\end{proof}

\begin{prop}\label{prop: f_P long first row}
Let $\l\in\mathcal{L}$, so that $r:=n-\l_1\le \log n$. Let $t=\frac n2(\log n-\frac 12\log l+c)$ for some absolute constant $c$. Then, we have
$$\left(\frac 1n+\frac{n-1}{n}\frac{\chi_\l(\tau)}{d_\l}\right)^t=(1+O(n^{-1+o(1)}))\exp(-2rt/n).$$
\end{prop}

\begin{proof}
Following the estimate given in \eqref{eq: long first row}, we have
$$\frac 1n+\frac{n-1}{n}\frac{\chi_\l(\tau)}{d_\l}=\exp(-2r/n+O(n^{-2+o(1)})).$$
Therefore, we have
$$\left(\frac 1n+\frac{n-1}{n}\frac{\chi_\l(\tau)}{d_\l}\right)^t=\exp(-2rt/n+O(tn^{-2+o(1)}))=(1+O(n^{-1+o(1)}))\exp(-2rt/n).$$
\end{proof}

\begin{lemma}\label{lem: dimension bound}
Let $\l\vdash n$, and $r:=n-\l_1$. Then, we have
\begin{enumerate}
\item $d_\lambda\le\binom{n}{r}\cdot d_{\lambda^*}\le\binom{n}{r}\cdot \sqrt{r!}\le n^r/\sqrt{r!}.$
\item When $r\le \log n$, we have 
$$d_\l=\binom{n}{r}\cdot d_{\l^*}\left(1-\frac{r}{n}+O(n^{-2+o(1)})\right).$$
\end{enumerate}
\end{lemma}

\begin{proof}
This is the same as \cite[Lemma 3.3]{jain2024hitting}.
\end{proof}

\begin{lemma}\label{lem: estimate of Kostka}
Let $\l\vdash n$, and $r:=n-\l_1,r':=n-\lambda_1'$. Then, we have
\begin{enumerate}
\item If $\l_1'>m$, then $K_{\lambda,(l^m)}=0$.
\item $K_{\lambda,(l^m)}\le m^r$. 
\item $K_{\lambda,(l^m)}\le\binom{n}{r}K_{\lambda^*,(\lfloor r/m\rfloor+1,\ldots,\lfloor r/m\rfloor+1,\lfloor r/m\rfloor,\ldots,\lfloor r/m\rfloor)}$. Here and from now on, for the partition $(\lfloor r/m\rfloor+1,\ldots,\lfloor r/m\rfloor+1,\lfloor r/m\rfloor,\ldots,\lfloor r/m\rfloor)$, there are $r-m\lfloor r/m\rfloor$ parts of size $\lfloor r/m\rfloor+1$, and $m\lfloor r/m\rfloor+m-r$ parts of size $\lfloor r/m\rfloor$.
\item $K_{\lambda,(l^m)}\le\binom{n}{r'}K_{\lambda^\bullet,(\lfloor r'/m\rfloor+1,\ldots,\lfloor r'/m\rfloor+1,\lfloor r'/m\rfloor,\ldots,\lfloor r'/m\rfloor)}$. Here and from now on, for the partition $(\lfloor r'/m\rfloor+1,\ldots,\lfloor r'/m\rfloor+1,\lfloor r'/m\rfloor,\ldots,\lfloor r'/m\rfloor)$, there are $r'-m\lfloor r'/m\rfloor$ parts of size $\lfloor r'/m\rfloor+1$, and $m\lfloor r'/m\rfloor+m-r'$ parts of size $\lfloor r'/m\rfloor$.
\end{enumerate}
\end{lemma}

\begin{proof}
Recall that by definition of Kostka numbers, $K_{\lambda,(l^m)}$ is the number of semistandard Young tableaux of shape $\lambda$ and weight $(l^m)$. The first item holds because there does not exist such a semistandard Young tableau. The second item holds because there are at most $m^r$ ways to fill in the second row and below of the tableau, and after that the first row is fully determined.

For the third item, let $\lambda^*:=(\lambda_2,\lambda_3,\ldots)$, so that $|\lambda^*|=r$.
Given a semistandard Young tableau $T$ of shape $\lambda$ and weight $(l^m)$, delete its first row and denote the resulting tableau by $T^*$. Then $T^*$ has shape $\lambda^*$ and some weight
\[
\alpha=(\alpha_1,\ldots,\alpha_m),\qquad 0\le \alpha_i\le l,\qquad \sum_{i=1}^m \alpha_i=r.
\]
Conversely, once $T^*$ is fixed, the first row of $T$ must contain exactly $l-\alpha_i$ copies of $i$ for each $1\le i\le m$. Since the first row is weakly increasing, there is at most one possible way to fill it. Hence
\[
K_{\lambda,(l^m)}
\le \sum_{\alpha} K_{\lambda^*,\alpha},
\]
where the sum runs over all weak compositions $\alpha=(\alpha_1,\ldots,\alpha_m)$ of $r$ satisfying $0\le \alpha_i\le l$.

Now let $\alpha^+$ be the partition obtained by rearranging the parts of $\alpha$ in weakly decreasing order. By the usual symmetry of Kostka numbers with respect to permuting the parts of the weight, we have
\[
K_{\lambda^*,\alpha}=K_{\lambda^*,\alpha^+}.
\]
Let
\[
\alpha_0:=\bigl(\lfloor r/m\rfloor+1,\ldots,\lfloor r/m\rfloor+1,\lfloor r/m\rfloor,\ldots,\lfloor r/m\rfloor\bigr),
\]
where there are $r-m\lfloor r/m\rfloor$ parts equal to $\lfloor r/m\rfloor+1$.
Then every partition $\alpha^+$ of $r$ with at most $m$ parts dominates $\alpha_0$. By the standard fact (see \cite[Chapter 2]{MR1824028}) that if $\mu\trianglerighteq \alpha_0$, then $h_{\alpha_0}-h_\mu$ is Schur-positive, it follows that
\[
K_{\lambda^*,\alpha^+}\le K_{\lambda^*,\alpha_0}.
\]
Therefore each summand above is bounded by $K_{\lambda^*,\alpha_0}$, and so
\[
K_{\lambda,(l^m)}
\le \left|\Bigl\{\alpha\in \mathbb{Z}_{\ge 0}^m:\ \sum_{i=1}^m \alpha_i=r,\ \alpha_i\le l\Bigr\}\right|\cdot K_{\lambda^*,\alpha_0}.
\]

It remains to bound the number of such $\alpha$. This number is exactly $[x^r](1+x+\cdots+x^l)^m$, i.e., the coefficient of $x^r$ in $(1+x+\cdots+x^l)^m$. Since coefficientwise
\[
1+x+\cdots+x^l \le (1+x)^l,
\]
we get
\[
[x^r](1+x+\cdots+x^l)^m \le [x^r](1+x)^{ml}=[x^r](1+x)^n=\binom{n}{r}.
\]
Hence
\[
K_{\lambda,(l^m)}\le \binom{n}{r}\,K_{\lambda^*,\alpha_0},
\]
which proves the third item.

The fourth item can be proved in an analogous way by considering the partition removing the first column.
\end{proof}

\begin{lemma}\label{lem: upper bound of dK}
We have 
$$\sum_{\l_1=n-r}d_\l K_{\l,(l^m)}\le\binom nr^2\cdot m^r,$$
$$\sum_{\l_1'=n-r'}d_\l K_{\l,(l^m)}\le\binom {n}{r'}^2\cdot m^{r'}.$$
\end{lemma}

\begin{proof}
We only prove the first inequality using the third item of \Cref{lem: estimate of Kostka}, since the second inequality follows analogously from the fourth item of \Cref{lem: estimate of Kostka}. Following the third item of \Cref{lem: estimate of Kostka}, we have
\begin{align}
\begin{split}
\sum_{\l_1=n-r}d_\l K_{\l,(l^m)}&\le\sum_{\lambda^*\vdash r}\binom nr^2d_{\lambda*}K_{\lambda^*,(\lfloor r/m\rfloor+1,\ldots,\lfloor r/m\rfloor+1,\lfloor r/m\rfloor,\ldots,\lfloor r/m\rfloor)}\\
&=\binom nr^2\cdot\frac{r!}{((\lfloor r/m\rfloor+1)!)^{r-m\lfloor r/m\rfloor}\cdot(\lfloor r/m\rfloor!)^{m\lfloor r/m\rfloor+m-r}}.\\
\end{split}
\end{align}
Therefore, we only need to prove
$$m^r\cdot((\lfloor r/m\rfloor+1)!)^{r-m\lfloor r/m\rfloor}\cdot(\lfloor r/m\rfloor!)^{m\lfloor r/m\rfloor+m-r}\ge r!.$$
In fact, we have
$$m^{\lfloor r/m\rfloor+1}(\lfloor r/m\rfloor+1)!\ge (r-j)(r-j-m)\cdots(r-j-\lfloor r/m\rfloor m)$$
for all $0\le j\le r-1-m\lfloor r/m\rfloor$, and
$$m^{\lfloor r/m\rfloor}(\lfloor r/m\rfloor)!\ge (m\lfloor r/m\rfloor-j)(m\lfloor r/m\rfloor-j-m)\cdots(\lfloor r/m\rfloor-j)$$
for all $0\le j\le m\lfloor r/m\rfloor+m-r-1$. Multiplying all these inequalities give the proof.
\end{proof}

The following bound generalizes \cite[Lemma 3.4]{jain2024hitting}. It shows that the sum $\sum_{\lambda\vdash n}d_\l K_{\l,(l^m)}\left|\frac 1n+\frac{n-1}{n}\frac{\chi_\lambda(\tau)}{d_\lambda}\right|^{2t}$ is mainly contributed by partitions in $\mathcal{L}$. 

\begin{thm}\label{thm: removal for P}
Let $t=\frac n2(\log n-\frac 12\log l+c)$. Then, we have
$$\sum_{\substack{\lambda\vdash n\\\lambda\notin \mc{L}}}d_\l K_{\lambda,(l^m)}\left|\frac 1n+\frac{n-1}{n}\frac{\chi_\lambda(\tau)}{d_\lambda}\right|^{2t}=n^{-\omega(1)}.$$
\end{thm}

\begin{proof}
We divide the sum into several cases and bound them separately. 
\begin{enumerate}
\item Suppose $\lambda_1\le 0.999n$ and $\lambda_1'\le 0.999n$. We first deal with the case $\lambda_1\ge\lambda_1'$. Let $r:=n-\l_1\ge0.001 n$. Following \eqref{eq: short first row}, we have 
$$\left|\frac 1n+\frac{n-1}{n}\frac{\chi_\lambda(\tau)}{d_\lambda}\right|\le1-
\frac rn\le\exp(-(1+10^{-100})r/n).$$
Also, applying \Cref{lem: upper bound of dK}, we have
$$\sum_{\l_1=n-r\ge\l_1'}d_\l K_{\l,(l^m)}\le\binom nr^2\cdot m^r=\exp(O(n)+r\log m).$$
Therefore, we have
\begin{align}
\begin{split}
\sum_{\l_1=n-r\ge\l_1'}d_\l K_{\lambda,(l^m)}\left|\frac 1n+\frac{n-1}{n}\frac{\chi_\lambda(\tau)}{d_\lambda}\right|^{2t}&\le\exp\left(-2(1+10^{-100})rt/n+O(n)+r\log m\right)\\
&\le\exp(-2\cdot10^{-100}rt/n+O(n))\\
&\le\exp(-2^{-1}\cdot10^{-100}r\log n+O(n))\\
&=n^{-\omega(1)}.
\end{split}
\end{align}
Here, on the third line, we are using $t/n=\frac12(\log n-\frac 12\log l+c)\ge \frac 14\log n+O(1)$. Summing over all $r\ge 0.001 n$ gives $n^{-\omega(1)}$. The case $\lambda_1'\ge\lambda_1$ follows similarly from the second inequality in .
\item Suppose $\lambda_1'>0.999n$. Since $m=n/l\le 0.5n$, we have $K_{\l,(l^m)}=0$ from the first item of \Cref{lem: estimate of Kostka}. Therefore, the total sum is zero.
\item Suppose $0.999n<\lambda_1<n-\log n$. Let $r:=n-\l_1\in[\log n,0.001n]$. Following \eqref{eq: middle first row}, we have 
$$\left|\frac 1n+\frac{n-1}{n}\frac{\chi_\lambda(\tau)}{d_\lambda}\right|\le1-2\frac rn+2\frac{r^2}{n^2}\le\exp\left(-2\frac rn+2\frac{r^2}{n^2}\right).$$
Let $p(r)$ be the number of partitions of size $r$. The Hardy-Ramanujan asymptotic gives $p(r)\lesssim\exp(\pi\sqrt{2r/3})$. Therefore, it follows from the first item of \Cref{lem: dimension bound} and the second item of \Cref{lem: estimate of Kostka} that
\begin{align}
\begin{split}
\sum_{\l_1=n-r}d_\l K_{\lambda,(l^m)}\left|\frac 1n+\frac{n-1}{n}\frac{\chi_\lambda(\tau)}{d_\lambda}\right|^{2t}&\le p(r)\frac{n^rm^r}{\sqrt{r!}}\exp\left(n\left(\log n-\frac 12\log l+c\right)\left(-2\frac rn+2\frac{r^2}{n^2}\right)\right)\\
&\le\exp(\pi\sqrt{2r/3})\cdot\exp\left(-\frac 12r\log r+2\frac {r^2}{n}\log n+O(r)\right)\\
&\le\exp\left(-\frac 13 r\log r\right)\\
&=n^{-\omega(1)}.
\end{split}
\end{align}
\end{enumerate}
Here, the third lines holds because $\frac{\log n}{n}\le 0.01\frac{\log r}{r}$. Combining the above cases together gives the proof.
\end{proof}

\subsection{Character estimates for $\nu_n^t$}

Following \cite[Section 3,3]{jain2024hitting}, for all $0\le M\le n/3$, we define a probability measure $\xi_M$ on $\mathfrak{S}_n$ as follows:
choose a subset $S\subset[n]$ uniformly among all subsets of size $M$, then sample a permutation
$\pi$ uniformly from $\S_{[n]\setminus S}$, and extend it to an element of $\mathfrak{S}_n$ 
by fixing every point of $S$. The following lemma gives an expression for the Fourier transform of $\xi_M$.

\begin{thm}[Fourier transform of $\xi_M$ and $\xi_M^c$]
\label{thm :zetaM_fourier_An}
Suppose $0\le M\le n/3$. Then, we have
\[
\widehat{\xi_M}(\lambda)
=
\frac{K_{\lambda,\mu}}{d_\lambda}\,\Id_{d_\lambda},
\]
where
\[
\mu=\mu_M:=(n-M,1,1,\dots,1)\vdash n,
\]
and $K_{\alpha,\mu}$ denotes the Kostka number.

Moreover, if $\l_1<n-M$, then $\widehat{\xi_M}(\lambda)=0$. If $\l_1\ge n-M$, we have
$$\widehat{\xi_M}(\lambda)=\binom{M}{\,n-\lambda_1\,}\frac{
d_{\lambda^*}}{d_\lambda}\Id_{d_\lambda}.$$
\end{thm}

\begin{proof}
The proof follows the same as \cite[Lemma 3.5]{jain2024hitting}.
\end{proof}

\begin{thm}\label{thm: character estimate for nu}
Let $\lambda\vdash n$, and $r:=n-\l_1$. Suppose $t=\frac n2(\log n-\frac 12\log l+c)$. Let $f_\nu(\l)$ be the scalar of $\widehat{\nu_n^t}$ at $\lambda$, i.e., $\widehat{\nu_n^t}(\l)=f_\nu(\l)\Id_{d_\l}$. Then, we have 
$$\left|f_\nu(\l)\right|\le\left(1-\frac 1n\right)^{-1}\cdots\left(1-\frac{r-1}{n}\right)^{-1}\exp\left(-2rt/n\right).$$
Suppose furthermore that $\lambda\in\mathcal{L}$, so that $r\le \log n$. Then, we have
$$f_\nu(\l)=\left(1-\frac rn+O(n^{-2+o(1)})\right)\exp\left(-2rt/n\right).$$
\end{thm}

\begin{proof}
Notice that
$$\nu_n^t=
\sum_{k=0}^n\xi_k\cdot\frac{\mathbf{P}[\Pois(\gamma_t)=k]}{\mathbf{P}[\Pois(\gamma_t)\le n]},$$
where $\xi_k$ is the distribution in the statement of \Cref{thm :zetaM_fourier_An}. Following the Fourier transform given in \Cref{thm :zetaM_fourier_An}, we have
\begin{align}
\begin{split}
f_\nu(\l)&=\mathbf{P}[\Pois(\gamma_t)\le n]^{-1}\cdot\sum_{k=0}^n\frac{d_{\lambda^*}\binom{k}{r}}{d_\lambda}\mathbf{P}[\Pois(\gamma_t)=k]\\
&\le\mathbf{P}[\Pois(\gamma_t)\le n]^{-1}\sum_{k=0}^n\binom{k}{r}\binom{n}{r}^{-1}\frac{\gamma_t^k e^{-\gamma_t}}{k !}\\
&=\left(1-\frac 1n\right)^{-1}\cdots\left(1-\frac{r-1}{n}\right)^{-1}\mathbf{P}[\Pois(\gamma_t)\le n]^{-1}\sum_{k=r}^nn^{-r}\cdot\frac{\gamma_t^k e^{-\gamma_t}}{(k-r)!}\\
&=\left(1-\frac 1n\right)^{-1}\cdots\left(1-\frac{r-1}{n}\right)^{-1}n^{-r}\cdot\gamma_t^r\cdot\frac{\mathbf{P}[\Pois
(\gamma_t)\le n-r]}{\mathbf{P}[\Pois(\gamma_t)\le n]}\\
&\le\left(1-\frac 1n\right)^{-1}\cdots\left(1-\frac{r-1}{n}\right)^{-1}\exp\left(-2rt/n\right).
\end{split}
\end{align}
Moreover, when $r\le\log n$, we have $\mathbf{P}[\Pois(\gamma_t)>n-r]\le\mathbf{P}[\Pois(n^{2/3})>n-r]=n^{-\omega(1)}$. Thus, a  refinement of the above gives
\begin{align}
\begin{split}
f_\nu(\l)&=\mathbf{P}[\Pois(\gamma_t)\le n]^{-1}\cdot\sum_{k=0}^n\frac{d_{\lambda^*}\binom{k}{r}}{d_\lambda}\mathbf{P}[\Pois(\gamma_t)=k]\\
&=\left(1-\frac rn+O(n^{-2+o(1)})\right)\sum_{k=0}^n\binom{k}{r}\cdot r!n^{-r}\cdot\frac{\gamma_t^k e^{-\gamma_t}}{k !}\\
&=\left(1-\frac rn+O(n^{-2+o(1)})\right)\sum_{k=r}^nn^{-r}\cdot\frac{\gamma_t^k e^{-\gamma_t}}{(k-r)!}\\
&=\left(1-\frac rn+O(n^{-2+o(1)})\right)n^{-r}\cdot\gamma_t^r\cdot\mathbf{P}[\Pois(\gamma_t)\le n-r]\\
&=\left(1-\frac rn+O(n^{-2+o(1)})\right)\exp\left(-2rt/n\right).
\end{split}
\end{align}
Here, the second line comes from the second item of \Cref{lem: dimension bound}. This completes the proof.
\end{proof}

\begin{thm}\label{thm: removal for nu}
Suppose $t=\frac n2(\log n-\frac 12\log l+c)$, and $\widehat{\nu_n^t}(\l)=f_\nu(\l)\Id_{d_\l}$. Then, we have
$$\sum_{\substack{\lambda\vdash n\\\lambda\notin \mc{L}}}d_\lambda K_{\lambda,(l^m)}|f_\nu(\lambda)|^2=n^{-\omega(1)}.$$
\end{thm}

\begin{proof}
Set $r:=n-\l_1$. When $r\in[\log n,0.9n]$, we have
$$\left(1-\frac 1n\right)^{-1}\cdots\left(1-\frac{r-1}{n}\right)^{-1}\le 10^r.$$
Therefore, it follows from the first item of \Cref{lem: dimension bound} and the second item of \Cref{lem: estimate of Kostka} that
\begin{align}
\begin{split}\sum_{n-\l_1=r}d_\lambda K_{\lambda,(l^m)}|f_\nu(\lambda)|^2&\le p(r)\cdot\binom nr\sqrt{r!}\cdot m^r\cdot\left(1-\frac 1n\right)^{-2}\cdots\left(1-\frac{r-1}{n}\right)^{-2}\exp\left(-4rt/n\right)\\
&\le \frac{p(r)10^r}{\sqrt{r!}}n^rm^r\exp\left(-4rt/n\right)\\
&\lesssim \exp(\pi\sqrt{2r/3}-\frac 12 r\log r+O(r))\\
&=n^{-\omega(1)}.
\end{split}
\end{align}
Summing over all $r\in[\log n,0.9n]$ gives $n^{-\omega(1)}$.
When $r\in[0.9n,n]$, by \Cref{lem: upper bound of dK} we have
\begin{align}
\begin{split}\sum_{n-\l_1=r}d_\lambda K_{\lambda,(l^m)}|f_\nu(\lambda)|^2&\le \binom nr^2\cdot m^r\cdot\left(1-\frac 1n\right)^{-2}\cdots\left(1-\frac{r-1}{n}\right)^{-2}\exp\left(-4rt/n\right)\\
&\le \frac{1}{(r!)^2}n^{2r}m^r\exp\left(-4rt/n\right)\\
&=\frac{n^r}{(r!)^2}\exp(-2cr)\\
&=n^{-\omega(1)}.
\end{split}
\end{align} 
Summing over all $r\in[0.9n,n]$ gives $n^{-\omega(1)}$. This completes the proof.
\end{proof}

\subsection{Proof of \Cref{thm: restate of pnkt and nu}}

We already have the necessary preparation to prove \Cref{thm: restate of pnkt and nu}. 

\begin{proof}[Proof of \Cref{thm: restate of pnkt and nu}]

For all $\l\vdash n$, denote 
$$f_P(\l):=\left(\frac 1n+\frac{n-1}{n}\frac{\chi_\lambda(\tau)}{d_\lambda}\right)^t,$$ 
so that $\widehat {P_n^{*t}}(\l)=f_P(\l)\Id_{d_\lambda}$. Let $f_\nu(\l)$ be defined the same way as in \Cref{thm: character estimate for nu}. Using the Cauchy-Schwarz inequality followed by \Cref{prop: quotient L2 formula}, we have that (the dagger superscript refers to the conjugate transpose matrix)
\begin{align}
\begin{split}
4\cdot d_{TV}(\sim_l\backslash P_n^{*t},\sim_l\backslash\nu_n^t)^2&\le \frac{n!}{(l!)^m}\sum_{\tilde\sigma\in\sim_l\backslash\S_n}\left(\sum_{\varsigma\in\tilde\sigma}(P_n^{*t}(\varsigma)-\nu_n^t(\varsigma))\right)^2\\
&=\sum_{\l\vdash n}d_\l K_{\lambda,(l^m)}|f_P(\l)-f_\nu(\l)|^2\\
&\le\sum_{\l\in\mathcal{L}}d_\l K_{\lambda,(l^m)}|f_P(\l)-f_\nu(\l)|^2+2\sum_{\l\notin\mathcal{L}}\left(|f_P(\l)|^2+|f_\nu(\l)|^2\right)\\
&=n^{-\omega(1)}+\sum_{\l\in\mathcal{L}}d_\l K_{\lambda,(l^m)}|f_P(\l)-f_\nu(\l)|^2.\\
\end{split}
\end{align}
Here, the final line follows from the bounds  given in \Cref{thm: removal for P} and \Cref{thm: removal for nu}. Now, we only need to prove
$$\sum_{\l\in\mathcal{L}}d_\l K_{\lambda,(l^m)}|f_P(\l)-f_\nu(\l)|^2\le n^{-2+o(1)}.$$ 
Applying the estimates given in \Cref{prop: f_P long first row} and \Cref{thm: character estimate for nu}, we have
$$|f_P(\l)-f_\nu(\l)|=n^{-1+o(1)}\exp(-2rt/n)$$
for all $\l\in\mathcal{L}$. Therefore, we have 
\begin{align}
\begin{split}
\sum_{\l\in\mathcal{L}}d_\l K_{\lambda,(l^m)}|f_P(\l)-f_\nu(\l)|^2&=\sum_{r=0}^{\log n}\sum_{n-\l_1=r}d_\l K_{\lambda,(l^m)}|f_P(\l)-f_\nu(\l)|^2\\
&\le n^{-2+o(1)}\cdot\sum_{r=0}^{\log n} p(r)\cdot\frac{n^r}{\sqrt{r!}}m^r\exp(-4rt/n)\\
&=n^{-2+o(1)}\cdot\sum_{r=0}^{\log n}\exp(\pi\sqrt{2r/3})\cdot\frac{\exp(-2cr)}{\sqrt{r!}}\\
&=n^{-2+o(1)}.
\end{split}
\end{align}
This completes the proof.
\end{proof}

\section{From tractable measure to total variation distance}\label{sec: From tractable measure to total variation distance}

In this section, we prove the following theorem, which elaborates \Cref{thm: limiting profile_main thm} in detail.

\begin{thm}\label{thm: restate of limiting profile}
Let $c\in\R$ be an absolute constant. Let $(n_N)_{N\ge 1},(m_N)_{N\ge 1}$ and $(l_N)_{N\ge 1}$ be sequences of positive integers, such that 
\begin{enumerate}
\item $m_N,l_N\ge 2$ for all $N\ge 1$, and $\lim_{N\rightarrow\infty}l_N=\infty$.
\item $n_N=m_Nl_N$ for all $N\ge 1$.
\end{enumerate}
For all $N\ge 1$, let $t_N=\frac{n_N}{2}(\log n_N-\frac 12\log l_N+c)$. Let $\sim_{l_N}\backslash P_{n_N}^{*t_N}$ be defined as in \Cref{defi: Pnk and its convolution}, and let $\sim_{l_N}\backslash U_{\S_{n_N}}$ be defined as in \Cref{defi: quotient of space}. Then, we have
$$\lim_{N\rightarrow\infty}d_{TV}(\sim_{l_N}\backslash P_{n_N}^{*t_N},\sim_{l_N}\backslash U_{\S_{n_N}})=2\Phi\left(\frac{\exp(-c)}{2}\right)-1$$
when $\lim_{N\rightarrow\infty}m_N=\infty$, and 
$$\lim_{N\rightarrow\infty}d_{TV}(\sim_{l_N}\backslash P_{n_N}^{*t_N},\sim_{l_N}\backslash U_{\S_{n_N}})=2\Phi\left(\sqrt{1-\frac 1m}\frac{\exp(-c)}{2}\right)-1$$
when $m_N=m\ge 2$ is fixed. 
\end{thm}

As in \Cref{sec: Approximating the shuffling with an explicitly tractable measure}, for notational convenience, we suppress the subscript $N$ for the rest of this section.

\begin{defi}
Let $n=ml$. We define the \emph{fixed point set} $F_{n,l}:\S_n\rightarrow 2^{[n]}$ as follows. For $\sigma\in\S_n$, let
$$F_{n,l}(\sigma):=\{i\in [n]:\ \sigma(i)\equiv i \pmod m\}.$$
The function $F_{n,l}$ naturally descends to a well-defined function on $\sim_l\backslash\S_n$, which we still denote by $F_{n,l}$. 
\end{defi}

For \(1\le i,j\le m\) and $\sigma\in\S_n$, define 
\[
X_{ij}=X_{ij}(\sigma):=\bigl|\{a\in T_{i,l}:\sigma(a)\in T_{j,l}\}\bigr|.
\]
Roughly speaking, $X_{ij}$ records how many cards of type $i$ are sent by the shuffle to the positions occupied by cards of type $j$. Then, we must have $\sum_{i=1}^m X_{ii}(\sigma)=F_{n,l}(\sigma)$ for all $\sigma\in\S_n$. We also write
$$X(\sigma):=(X_{ij}(\sigma))_{1\le i,j\le m}\in\Mat_m(\Z_{\ge 0}),\quad\diag(\sigma)=(X_{11}(\sigma),\ldots,X_{mm}(\sigma))\in\Z_{\ge 0}^m,\quad\forall\sigma\in\S_n.$$
It is clear that every row and column of $X(\sigma)$ has sum $l$. Similarly to $F_{n,l}(\cdot)$, the functions $\diag(\cdot)$ and $X(\cdot)$ also descend to  $\sim_l\backslash\S_n$, where we will use the same symbol for these functions. 

When the construction of the subset $S_t\subset [n]$ is clear from the text, we write 
$$R_i:=|S_t\cap T_{i,l}|,\quad\forall 1\le i\le m,$$
so that $\sum_{i=1}^m R_i=|S_t|$.

In order to simplify the computations, we introduce the following truncated version $\mu_n^t$ of $\nu_n^t$.

\begin{defi}\label{defi: construction of mu}
Suppose $t\in\Z$. Set $t':=t-\lfloor\frac{n\log n}{2}\rfloor$, and $\gamma_t:=\exp(-2t'/n)$, which is the same as in \Cref{defi: construction of nu}. Let $\mu_n^t$ denote the measure on $\S_n$ defined by the following sampling process: first, sample $M_t\in\{0,1,\ldots,n\}$ according to the distribution 
$$\mathbf{P}[M_t=k]=\begin{cases}\mathbf{P}(\Pois(\gamma_t)=k)/\mathbf{P}(\gamma_t-l^{1/3}\le\Pois(\gamma_t)\le \gamma_t+l^{1/3})&\gamma_t-l^{1/3}\le k\le \gamma_t+l^{1/3}\\
0 & \text{else}
\end{cases},$$ 
then sample a uniformly random subset $S_t$ of size $M_t$ in $[n]$. Finally, sample a uniformly random element of $\S_{[n]\backslash S_t}$ and view it as an element of $\S_n$ by fixing all of the elements in $S_t$.
\end{defi}

\begin{prop}\label{prop: transfer to diag}
Let $n=ml$, such that $m,l\ge 2$. Suppose 
$t=\frac n2\left(\log n-\frac 12\log l+c\right)$,  where $c\in \R$ is an absolute constant. Then, we have 
$$d_{TV}(\sim_l\backslash P_{n}^{*t},\sim_l\backslash U_{\S_n})=d_{TV}(\diag(\sigma),\diag(\tau))+o(1).$$
Here, $\tau\in\S_n$ be distributed with respect to $U_{\S_n}$, and $\sigma\in\S_n$ be distributed with respect to $\mu_n^t$.
\end{prop}

\begin{proof}
Let $\tau\in\S_n$ be distributed with respect to $U_{\S_n}$, and $\sigma\in\S_n$ be distributed with respect to $\mu_n^t$. Since $\gamma_t=\exp(-c)\sqrt{l}$, it follows naturally from the scaling local limit of poisson distribution that
$$d_{TV}(\sim_l\backslash \mu_n^t,\sim_l\backslash \nu_n^t)\le d_{TV}(\mu_n^t,\nu_n^t)=o(1).$$
Therefore, following \Cref{thm: pnkt and nu}, it suffices to prove 
$$d_{TV}(\sim_l\backslash \mu_n^t,\sim_l\backslash U_{\S_n})=d_{TV}(\diag(\sigma),\diag(\tau)).$$
Observe that for all fixed $(a_{11},\ldots,a_{mm})\in\Z_{\ge 0}^m$ such that $\mathbf{P}(\diag(\tau)=(a_{11},\ldots,a_{mm}))>0$, the conditional law of $\sim_l\backslash\tau$ under the event $\diag(\tau)=(a_{11},\ldots,a_{mm})$ is uniform on the set 
$$\{\sim_l\backslash\tau\in\sim_l\backslash\S_n:\diag(\tau)=(a_{11},\ldots,a_{mm})\}.$$
Also, the same uniform assertion holds for the random permutation $\sigma$. Since the conditional laws of $\sim_l\backslash\tau$ and $\sim_l\backslash\sigma$ coincide on every fiber of $\diag(\cdot)$, the total variation distance between their laws is completely determined by the masses assigned to these fibers. Equivalently, passing to the quotient space determined by $\diag(\cdot)$ does not change the total variation distance, and hence
$$
d_{TV}(\sim_l\backslash \mu_n^t,\sim_l\backslash U_{\S_n})
=
d_{TV}(\diag(\sigma),\diag(\tau)).
$$
This proves the proposition. 
\end{proof}

\begin{lemma}\label{lem: distribution of X}
Let $n=ml$, such that $m,l\ge 2$. Suppose the entries of $A=(a_{ij})\in\Mat_m(\Z)$ are nonnegative, and all the rows and columns have sum $l$. Let $\tau$ be distributed with respect to $U_{\S_n}$. Then, we have
$$\mathbf{P}(X(\tau)=A)=\frac{(l!)^{2m}}{n!\cdot\prod_{1\le i,j\le m}a_{ij}!}.$$
Moreover, suppose 
$t=\frac n2\left(\log n-\frac 12\log l+c\right)$,  where $c\in \R$ is an absolute constant. Let $\sigma$ be distributed with respect to $\mu_n^t$. Then, for all $k\ge 0$, we have
$$\mathbf{P}(X(\sigma)=A\mid |S_t|=k)=\sum_{r_1+\cdots+r_m=k}\frac{k!(l!)^m\prod_{i=1}^m(l-r_i)!}{n!\prod_{i\ne j}a_{ij}!}\prod_{i=1}^m\binom{a_{ii}}{r_i}.$$
\end{lemma}

\begin{proof}
We first prove the formula for $\tau$. For each fixed $i$, partition the set $T_{i,l}$ into $m$ labeled pieces
\[
T_{i,l}=B_{i1}\sqcup\cdots\sqcup B_{im},
\qquad |B_{ij}|=a_{ij}.
\]
The number of ways to do this is again
\[
\frac{l!}{\prod_{j=1}^m a_{ij}!}.
\]
Multiplying over all $i$, the number of ways to decompose all domain blocks is
\[
\prod_{i=1}^m \frac{l!}{\prod_{j=1}^m a_{ij}!}
=\frac{(l!)^m}{\prod_{1\le i,j\le m} a_{ij}!}.
\]

Similarly, for each fixed $j$, partition the target block $T_{j,l}$ into $m$ labeled pieces
\[
T_{j,l}=C_{1j}\sqcup\cdots\sqcup C_{mj},
\qquad |C_{ij}|=a_{ij}.
\]
The number of ways to do this is
\[
\frac{l!}{\prod_{i=1}^m a_{ij}!},
\]
and hence the total number of ways to decompose all target blocks is
\[
\prod_{j=1}^m \frac{l!}{\prod_{i=1}^m a_{ij}!}
=\frac{(l!)^m}{\prod_{1\le i,j\le m} a_{ij}!}.
\]

Once these decompositions are fixed, for each pair $(i,j)$ we must choose a bijection
\[
B_{ij}\to C_{ij},
\]
and the number of such bijections is $a_{ij}!$. Therefore the total number of permutations
$\pi\in \S_n$ with $X(\pi)=A$ equals
\[
\frac{(l!)^m}{\prod_{i,j} a_{ij}!}\cdot
\frac{(l!)^m}{\prod_{i,j} a_{ij}!}\cdot
\prod_{i,j} a_{ij}!
=
\frac{(l!)^{2m}}{\prod_{i,j} a_{ij}!}.
\]
Since $\tau$ is uniform on $\S_n$, we conclude that
\[
\mathbf P(X(\tau)=A)
=
\frac{1}{n!}\cdot \frac{(l!)^{2m}}{\prod_{1\le i,j\le m} a_{ij}!}.
\]

We now prove the second formula. Consider nonnegative integers $r_1,\ldots,r_m$ with $r_1+\cdots+r_m=k$. Suppose the subset $S_t\subset[n]$ is already chosen, and that 
$$R_i:=|S_t\cap T_{i,l}|=r_i,\quad\forall1\le i\le m.$$
We now count the number of permutations $\sigma\in \S_{[n]\backslash S_t}$ such that $X(\sigma)=A$. Necessarily, among the $a_{ii}$ points contributing to $X_{ii}(\sigma)$, exactly $r_i$ of them are
the fixed points in $T_{i,l}$. Hence there are
$\binom{a_{ii}}{r_i}$ ways to choose which diagonal matches correspond to actual fixed points in block $T_{i,l}$. After these $r_i$ fixed points are chosen for each $i$, there remain $l-r_i$ non-fixed points
in the domain block $T_{i,l}$ and also $l-r_i$ non-fixed points in the target block $T_{i,l}$.
These remaining points must be matched according to the matrix obtained from $A$ by removing
$r_i$ from the $(i,i)$ entry for each $i$. Thus the effective matrix of non-fixed moves is
\[
A'=(a'_{ij}),\qquad
a'_{ij}=
\begin{cases}
a_{ii}-r_i,& i=j,\\
a_{ij},& i\ne j.
\end{cases}
\]
Applying the same counting argument as in the first part to this reduced matrix, the number of
such permutations is
\[
\left(\prod_{i=1}^m \binom{a_{ii}}{r_i}\right)
\cdot
\frac{(l!)^m\prod_{i=1}^m (l-r_i)!}
{\prod_{i\ne j} a_{ij}!\prod_{i=1}^m (a_{ii}-r_i)!}.
\]
Using
\[
\binom{a_{ii}}{r_i}=\frac{a_{ii}!}{r_i!(a_{ii}-r_i)!},
\]
one may also view this as first choosing the fixed points and then bijecting the remaining pieces;
however, the above form is the one that will match the desired probability formula after dividing
by the total number of possibilities.

Now, given $|S_t|=k$, the total number of possible outcomes in the construction of $\sigma$ is
\[
\binom{n}{k}(n-k)! = \frac{n!}{k!},
\]
since we choose the fixed set $S_t$ and then a permutation of its complement.
On the other hand, once $(r_1,\ldots,r_m)$ is fixed, the number of ways to choose the subset
$S_t$ so that $|S_t\cap T_{i,l}|=r_i$ for all $i$ is $\prod_{i=1}^m \binom{l}{r_i}$. For each such choice, the number of permutations of the complement realizing the matrix $A$
is exactly
\[
\frac{\prod_{i=1}^m (l-r_i)!^2}
{\prod_{i\ne j} a_{ij}!\prod_{i=1}^m (a_{ii}-r_i)!}.
\]
Therefore, summing over all $(r_1,\ldots,r_m)$ with total $k$, we get
\[
\mathbf P(X(\sigma)=A\mid |S_t|=k)
=
\frac{1}{\binom{n}{k}(n-k)!}
\sum_{r_1+\cdots+r_m=k}
\left(\prod_{i=1}^m \binom{l}{r_i}\right)
\frac{\prod_{i=1}^m (l-r_i)!^2}
{\prod_{i\ne j} a_{ij}!\prod_{i=1}^m (a_{ii}-r_i)!}.
\]
Since
\[
\binom{n}{k}(n-k)! = \frac{n!}{k!},
\qquad
\binom{l}{r_i}(l-r_i)! = \frac{l!}{r_i!},
\qquad
\frac{k!}{\prod_i r_i!}\prod_{i=1}^m \frac{a_{ii}!}{(a_{ii}-r_i)!}
=
k!\prod_{i=1}^m \binom{a_{ii}}{r_i},
\]
this simplifies exactly to
\[
\mathbf P(X(\sigma)=A\mid |S_t|=k)
=
\sum_{r_1+\cdots+r_m=k}
\frac{k!(l!)^m\prod_{i=1}^m(l-r_i)!}{n!\prod_{i\ne j}a_{ij}!}
\prod_{i=1}^m\binom{a_{ii}}{r_i}.
\]
This proves the lemma.
\end{proof}

In the following, we split the proof of \Cref{thm: restate of limiting profile} into the two cases: the fixed $m$ regime and the growing regime where $m=\omega(1)$. %We will write $\mathbf 1:=(1,\ldots,1)$ for the $m$-dimensional vector whose all entries are $1$.

\subsection{Proof for fixed $m$}

In this subsection, we let $m\ge 2$ be a fixed integer. 

% \begin{lemma}%\label{lem: tail bound S_t fixed m}
% Let $n=ml$, such that $m\ge 2$ is fixed and $l=\omega(1)$. Suppose $
% t=\frac n2\left(\log n-\frac12\log l+c\right)$, where $c\in \mathbb R$ is an absolute constant. Sample $S_t\subset[n]$ the same way as in
% \Cref{defi: construction of mu}. Then, for all $|k-\exp(-c)\sqrt{l}|<l^{1/3}$, we have
% \[
% \mathbf P\!\left(\max_{1\le i\le m}\left|R_i-\frac{k}{m}\right|>l^{1/3}\,\middle|\,|S_t|=k\right)=O(l^{-1/6}).
% \]
% \end{lemma}

% \begin{proof}
% Fix $1\le i\le m$. Conditional on the event $|S_t|=k$, the random variable $R_i$ has the
% hypergeometric distribution obtained by choosing $k$ points uniformly from $[n]$, among which
% exactly $l$ belong to $T_{i,l}$. Therefore
% \[
% \mathbb E[R_i\mid |S_t|=k]=k\frac{l}{n}=\frac{k}{m}, \quad\mathrm{Var}(R_i\mid |S_t|=k)
% =
% k\frac{l}{n}\left(1-\frac{l}{n}\right)\frac{n-k}{n-1}
% \le \frac{k}{m}.
% \]
% Thus
% \[
% \mathbf P\!\left(\left|R_i-\frac{|S_t|}{m}\right|>l^{1/3}\,\middle|\,|S_t|=k\right)
% \le
% \frac{\mathrm{Var}(R_i\mid |S_t|=k)}{l^{2/3}}
% \le
% \frac{k}{m\,l^{2/3}}=
% O(l^{-1/6})
% \]
% using the fact $k=O(l^{1/2})$. Since $m$ is fixed, a union bound over $1\le i\le m$ yields
% \[
% \mathbf P\!\left(\max_{1\le i\le m}\left|R_i-\frac{|S_t|}{m}\right|>l^{1/3}\right)=O(l^{-1/6}).
% \]
% This completes the proof.
% \end{proof}

\begin{lemma}\label{lem: tail bound and expectation fixed point fixed m}
Let $n=ml$, such that $m\ge 2$ is fixed and $l=\omega(1)$. Suppose $
t=\frac n2\left(\log n-\frac12\log l+c\right)$, where $c\in \R$ is an absolute constant. Let $\tau\in\S_n$ be distributed with respect to $U_{\S_n}$, and let
$\sigma\in\S_n$ be distributed with respect to $\mu_n^t$. Then, we have
\[
\E|F_{n,l}(\tau)|=l,\quad\mathbf P\!\left(\max_{1\le i\le m}\left|X_{ii}(\tau)-\frac lm\right|>l^{2/3}\right)=O(l^{-1/3}),
\]
and 
\[
\E|F_{n,l}(\sigma)|=l+\Bigl(1-\frac1m\Bigr)\exp(-c)\sqrt l+O(l^{1/3}),\quad\mathbf P\!\left(\max_{1\le i\le m}\left|X_{ii}(\sigma)-\frac lm\right|>l^{2/3}\right)=O(l^{-1/3}).
\]
\end{lemma}

\begin{proof}
We first prove the assertion for \(\tau\). Observe that
\[
\mathbb E[X_{ii}(\tau)]=\frac{l\cdot l}{n}=\frac lm,\quad \mathrm{Var}(X_{ii}(\tau))
=
l\cdot \frac{l}{n}\left(1-\frac{l}{n}\right)\frac{n-l}{n-1}
\le l.
\]
Therefore, we have $\E|F_{n,l}(\tau)|=\sum_{i=1}^m\mathbb E[X_{ii}(\tau)]=l$. Also, by Chebyshev's inequality,
\[
\mathbf P\!\left(\left|X_{ii}(\tau)-\frac lm\right|>l^{2/3}\right)
\le \frac{l}{l^{4/3}}=O(l^{-1/3}).
\]
Since \(m\) is fixed, taking a union bound over \(1\le i\le m\) yields
\[
\mathbf P\!\left(\max_{1\le i\le m}\left|X_{ii}(\tau)-\frac lm\right|>l^{2/3}\right)=O(l^{-1/3}).
\]

We next prove the assertion for \(\sigma\). Fix $k$ and $r_1,\ldots,r_m$ such that
$$
|k-\exp(-c)\sqrt{l}|<l^{1/3},\quad r_1+\cdots+r_m=k.
$$
Condition on \(|S_t|=k\) and \((R_1,\dots,R_m)=(r_1,\ldots,r_m)\). Then we have
\[
X_{ii}(\sigma)=r_i+Y_i,
\]
where \(Y_i=Y_i(\sigma)\) counts how many of the remaining \(l-r_i\) points in \(T_{i,l}\) are sent into the
remaining \(l-r_i\) target positions in \(T_{i,l}\). Since on \([n]\setminus S_t\) the permutation is
uniform, we have
\begin{equation*}%\label{eq: conditional fixed point over k and r}
\mathbb E[X_{ii}(\sigma)\mid |S_t|=k,(R_1,\dots,R_m)=(r_1,\ldots,r_m)]=
r_i+\frac{(l-r_i)^2}{n-k}
\end{equation*}
and
\[
\mathrm{Var}(X_{ii}(\sigma)\mid |S_t|=k,(R_1,\dots,R_m)=(r_1,\ldots,r_m))
=O(l).
\]
Therefore, we have
\begin{align}
\begin{split}
\E[|F_{n,l}(\sigma)||S_t|=k,(R_1,\dots,R_m)=(r_1,\ldots,r_m)]&=\sum_{i=1}^m \left(r_i+\frac{(l-r_i)^2}{n-k}\right)\\
&=k+\frac{nl-2kl}{n-k}+\sum_{i=1}^m\frac{r_i^2}{n-k}\\
&=l+(1-\frac 1m)k+O(1).
\end{split}
\end{align}
Since we always have $|k-\exp(-c)\sqrt{l}|<l^{1/3}$, we deduce that $\E|F_{n,l}(\sigma)|=l+\Bigl(1-\frac1m\Bigr)\exp(-c)\sqrt l+O(l^{1/3})$.
Also, following Chebyshev's inequality, for all $1\le i\le m$,
\begin{equation*}%\label{eq: tail for Xii}
\mathbf P\!\left(
\left|X_{ii}(\sigma)-\frac lm\right|> l^{2/3}
\Bigg| |S_t|=k,(R_1,\dots,R_m)=(r_1,\ldots,r_m)\right)=O\left(\frac{l}{l^{4/3}}\right)=O(l^{-1/3}).
\end{equation*}
Due to the uniformity above, we can remove the condition for $|S_t|$ and $R_1,\ldots,R_m$ and then sum over all $1\le i\le m$, which gives
\[
\mathbf P\!\left(
\max_{1\le i\le m}\left|X_{ii}(\sigma)-\frac lm\right|>l^{2/3}
\right)=O(l^{-1/3}).
\]
This proves the lemma.
\end{proof}

The following lemma gives the central limit theorem for the number of fixed points, which is a specific example of combinatorial central limit theorem from Hoeffding \cite{hoeffding1951combinatorial}.

\begin{lemma}\label{lem: fixed point CLT_fixed m}
Let $n=ml$, such that $m\ge 2$ is fixed and $l=\omega(1)$. Suppose $
t=\frac n2\left(\log n-\frac12\log l+c\right)$, where $c\in \R$ is an absolute constant. Let $\tau\in\S_n$ be distributed with respect to $U_{\S_n}$, and let
$\sigma\in\S_n$ be distributed with respect to $\mu_n^t$. Then, we have
$$
\frac{|F_{n,l}(\tau)|-l}{\sqrt{(1-1/m)l}}
\xrightarrow{d}\mathcal N(0,1),\quad\frac{|F_{n,l}(\sigma)|-l-(1-1/m)\exp(-c)\sqrt l}{\sqrt{(1-1/m)l}}
\xrightarrow{d}\mathcal N(0,1).
$$
\end{lemma} 

\begin{proof}
We first treat the uniform permutation $\tau$. For $1\le i,j\le n$, define
\[
c(i,j):=\mathbf 1_{\{i\equiv j \,(\mathrm{mod}\ m)\}}.
\]
Then
\[
|F_{n,l}(\tau)|=\sum_{i=1}^n c(i,\tau(i)).
\]
Thus $|F_{n,l}(\tau)|$ is a permutation statistic of the form considered by Hoeffding
\cite[Theorems 2 and 3]{hoeffding1951combinatorial}. For every fixed $i$, there are exactly $l=n/m$ values of $j$ such that $i\equiv j \pmod m$.
Hence
\[
\frac1n\sum_{j=1}^n c(i,j)=\frac1m,
\qquad
\frac1n\sum_{i=1}^n c(i,j)=\frac1m,
\qquad
\frac1{n^2}\sum_{i,j=1}^n c(i,j)=\frac1m.
\]
Therefore the centered array $d(i,j)$ in Hoeffding's notation is simply
\[
|d(i,j)|=\left|c(i,j)-\frac1m\right|\le 1.
\]
By Hoeffding's variance formula \cite[(10)]{hoeffding1951combinatorial},
\[
\Var(|F_{n,l}(\tau)|)
=
\frac1{n-1}\sum_{i,j=1}^n d(i,j)^2=\frac{n\,l(1-1/m)}{n-1}
=
(1-1/m)l+o(l),
\]
and therefore
\[
\frac{\max_{i,j}|d(i,j)|}{\sqrt{\Var(|F_{n,l}(\tau)|)}}\le\frac{1}{\sqrt{\Var(|F_{n,l}(\tau)|)}}=o(1)
\]
since $m$ is fixed and $l=\omega(1)$. Therefore, Hoeffding's combinatorial central limit theorem
\cite[Theorem 3 together with (12)]{hoeffding1951combinatorial} applies, and yields
\[
\frac{|F_{n,l}(\tau)|-l}{\sqrt{(1-1/m)l}}
\xrightarrow{d}\mathcal N(0,1).
\]
Here, we use the fact $\E|F_{n,l}(\tau)|=l$ from \Cref{lem: tail bound and expectation fixed point fixed m}.

We now turn to $\sigma$. Fix $k$ and $r_1,\ldots,r_m$ such that
$$
|k-\exp(-c)\sqrt{l}|<l^{1/3},\quad r_1+\cdots+r_m=k.
$$
Condition on \(|S_t|=k\) and \((R_1,\dots,R_m)=(r_1,\ldots,r_m)\).
Then
\[
|F_{n,l}(\sigma)|=k+Z,
\]
where $Z$ is the number of indices in $[n]\setminus S_t$ whose image under the random
permutation of $[n]\setminus S_t$ stays in the same card type. After relabeling the set
$[n]\setminus S_t$, this is again a permutation statistic of Hoeffding type:
\[
Z=\sum_{x} c_{k,R}(x,\pi(x)),
\]
where $\pi$ is uniform on the remaining $n-k$ points and
\[
c_{k,R}(x,y)=\mathbf 1_{\{\text{$x$ and $y$ belong to the same remaining type block}\}}.
\]
Applying Hoeffding's variance formula again, we have
\[
\Var(Z\mid k,R_1,\dots,R_m)
=
\frac1{n-k-1}\sum_{x,y} d_{k,R}(x,y)^2,
\]
where $d_{k,R}$ is the centered array attached to $c_{k,R}$. A routine computation shows that
\[
\Var(Z\mid k,R_1,\dots,R_m)
=
(1-1/m)l+o(l),
\]
and that
\[
\frac{\max_{x,y}|d_{k,R}(x,y)|}{\sqrt{\Var(Z\mid k,R_1,\dots,R_m)}}=o(1).
\]
Therefore, Hoeffding's theorem \cite[Theorem 3 together with (12)]{hoeffding1951combinatorial}
again yields
\[
\frac{|F_{n,l}(\sigma)|-l-(1-1/m)\exp(-c)\sqrt l}{\sqrt{(1-1/m)l}}
\xrightarrow{d}\mathcal N(0,1).
\]
Here, we use the fact $\E|F_{n,l}(\sigma)|=l+\Bigl(1-\frac1m\Bigr)\exp(-c)\sqrt l+O(l^{1/3})$ from \Cref{lem: tail bound and expectation fixed point fixed m}. This proves the lemma.
\end{proof}

\begin{lemma}\label{lem: concentration of St_fixed m}
Let $a_{11},\ldots,a_{mm},k$ be fixed, such that
$$\frac lm-l^{2/3}\le a_{11},\ldots,a_{mm}\le \frac lm+l^{2/3},$$
$$\exp(-c)\sqrt {l}-l^{1/3}\le k\le\exp(-c)\sqrt {l}+l^{1/3}.$$ 
Then, we have
$$(1-O(l^{-1/6}))\binom{\sum_{i=1}^m a_{ii}}{k}\le\sum_{\substack{r_1+\cdots+r_m=k\\|r_i-k/m|<l^{1/3},\forall i}}\prod_{i=1}^m\binom{a_{ii}}{r_i}\le\binom{\sum_{i=1}^m a_{ii}}{k}.$$
\end{lemma}

\begin{proof}
By Vandermonde's identity,
\[
\sum_{r_1+\cdots+r_m=k}\prod_{i=1}^m \binom{a_{ii}}{r_i}
=
\binom{\sum_{i=1}^m a_{ii}}{k}.
\]
Therefore the upper bound is immediate. It remains to prove the lower bound. Consider the following experiment: choose uniformly a
\(k\)-element subset of a set of size \(\sum_{i=1}^m a_{ii}\), which is partitioned into \(m\) disjoint
blocks of sizes \(a_{11},\dots,a_{mm}\). Let \(R_i\) denote the number of chosen elements from the
\(i\)-th block. Then \((R_1,\dots,R_m)\) has the multivariate hypergeometric distribution, and for
every \((r_1,\dots,r_m)\in\Z_{\ge0}^m\) with \(r_1+\cdots+r_m=k\),
\[
\mathbf P\bigl((R_1,\dots,R_m)=(r_1,\dots,r_m)\bigr)
=
\frac{\prod_{i=1}^m \binom{a_{ii}}{r_i}}{\binom{\sum_{i=1}^m a_{ii}}{k}}.
\]
Hence
\[
\frac{1}{\binom{\sum_{i=1}^m a_{ii}}{k}}
\sum_{\substack{r_1+\cdots+r_m=k\\ |r_i-k/m|<l^{1/3},\forall i}}
\prod_{i=1}^m \binom{a_{ii}}{r_i}
=
\mathbf P\Bigl(\max_{1\le i\le m}|R_i-k/m|<l^{1/3}\Bigr).
\]
So it suffices to show
\[
\mathbf P\Bigl(\max_{1\le i\le m}|R_i-k/m|\ge l^{1/3}\Bigr)=O(l^{-1/6}).
\]

Fix \(1\le i\le m\). Since \(R_i\) is hypergeometric, we have
\[
\mathbf E[R_i]=k\frac{a_{ii}}{\sum_{j=1}^m a_{jj}},
\qquad
\mathrm{Var}(R_i)
=
k\frac{a_{ii}}{\sum_{j=1}^m a_{jj}}
\left(1-\frac{a_{ii}}{\sum_{j=1}^m a_{jj}}\right)
\frac{\sum_{j=1}^m a_{jj}-k}{\sum_{j=1}^m a_{jj}-1}
\le
k\frac{a_{ii}}{\sum_{j=1}^m a_{jj}}.
\]
Now
\[
a_{ii}=\frac lm+O(l^{2/3}),
\qquad
\sum_{j=1}^m a_{jj}=l+O(l^{2/3}),
\]
so
\[
\frac{a_{ii}}{\sum_{j=1}^m a_{jj}}
=
\frac1m+O(l^{-1/3}).
\]
Since \(k=\exp(-c)\sqrt l+O(l^{1/3})\), it follows that
\[
\mathbf E[R_i]
=
\frac{k}{m}+O(kl^{-1/3})
=
\frac{k}{m}+O(l^{1/6}),
\]
and also
\[
\mathrm{Var}(R_i)\le k\frac{a_{ii}}{\sum_{j=1}^m a_{jj}}=O(\sqrt l).
\]

Therefore, for \(l\) sufficiently large,
\[
\left|\mathbf E[R_i]-\frac{k}{m}\right|\le \frac12 l^{1/3}.
\]
Hence
\[
\left\{\left|R_i-\frac{k}{m}\right|\ge l^{1/3}\right\}
\subseteq
\left\{\left|R_i-\mathbf E[R_i]\right|\ge \frac12 l^{1/3}\right\},
\]
and Chebyshev's inequality gives
\[
\mathbf P\left(\left|R_i-\frac{k}{m}\right|\ge l^{1/3}\right)
\le
\mathbf P\left(\left|R_i-\mathbf E[R_i]\right|\ge \frac12 l^{1/3}\right)
\le
\frac{4\,\mathrm{Var}(R_i)}{l^{2/3}}
=
O(l^{-1/6}).
\]
Since \(m\) is fixed, a union bound yields
\[
\mathbf P\Bigl(\max_{1\le i\le m}|R_i-k/m|\ge l^{1/3}\Bigr)=O(l^{-1/6}).
\]
Combining this with the previous Vandermonde's identity, we conclude that
\[
\sum_{\substack{r_1+\cdots+r_m=k\\ |r_i-k/m|<l^{1/3},\forall i}}
\prod_{i=1}^m \binom{a_{ii}}{r_i}
\ge
(1-O(l^{-1/6}))
\binom{\sum_{i=1}^m a_{ii}}{k},
\]
which proves the lemma.
\end{proof}

\begin{proof}[Proof of \Cref{thm: restate of limiting profile}, fixed $m$ case] 
Let $\tau\in\S_n$ be distributed with respect to $U_{\S_n}$, and let
$\sigma\in\S_n$ be distributed with respect to $\mu_n^t$. By \Cref{prop: transfer to diag}, it suffices to prove
$$d_{TV}(\diag(\sigma),\diag(\tau))=2\Phi\left(\sqrt{1-\frac 1m}\frac{\exp(-c)}{2}\right)-1+o(1).$$
Let $c_0>0$ be another absolute constant, and write
$$k_{c_0,c}:=(1/2+c_0)(1-1/m)\exp(-c)\sqrt l,\quad k_{-c_0,c}:=(1/2-c_0)(1-1/m)\exp(-c)\sqrt l.$$
We claim that the following holds when $l$ is sufficiently large. For all sequence $(a_{11},\ldots,a_{mm})$ such that
$\max_{1\le i\le m}|a_{ii}-l/m|\le l^{2/3}$, one has 
\begin{equation}\label{eq: large fixed point_fixed m}
\sum_{i=1}^m a_{ii}\ge l+k_{c_0,c}\text{ implies }\mathbf{P}(\diag(\sigma)=(a_{11},\ldots,a_{mm}))\ge \mathbf{P}(\diag(\tau)=(a_{11},\ldots,a_{mm})),
\end{equation}
\begin{equation}\label{eq: small fixed point_fixed m}
\sum_{i=1}^m a_{ii}\le l+k_{-c_0,c}\text{ implies }\mathbf{P}(\diag(\sigma)=(a_{11},\ldots,a_{mm}))\le\mathbf{P}(\diag(\tau)=(a_{11},\ldots,a_{mm})).
\end{equation}
We first deal with \eqref{eq: large fixed point_fixed m}. It suffices to prove that for all $|k-\exp(-c)\sqrt{l}|\le l^{1/3}$ and $\sum_{i=1}^m a_{ii}\ge l+k_{c_0,c}$ with $\max_{1\le i\le m}|a_{ii}-l/m|\le l^{2/3}$, one has
$$\mathbf{P}(\diag(\sigma)=(a_{11},\ldots,a_{mm})\mid |S_t|=k)\ge\mathbf{P}(\diag(\tau)=(a_{11},\ldots,a_{mm})).$$
Following the explicit expression given in \Cref{lem: distribution of X} and removing the overlapping terms, this is equivalent to
$$k!\sum_{r_1+\cdots+r_m=k}\prod_{i=1}^m(l)_{r_i}^{-1}\binom{a_{ii}}{r_i}\ge1.$$
Here, $(l)_{r_i}:=l(l-1)\cdots(l-r_i+1)$ denotes the falling factorial. In particular, when $|r_i-k/m|<l^{1/3}$ for all $1\le i\le m$, we have
\begin{align}\label{eq: estimate of falling factorials_fixed m}
\begin{split}
\prod_{i=1}^m (l)_{r_i}&=l^k\prod_{i=1}^m (1-1/l)\cdots(1-(r_i-1)/l)\\
&=l^k\exp\left(-\sum_{i=1}^m\left(1/l+\cdots+(r_i-1)/l+O(l^{-1/2})\right)\right)\\
&=l^k\exp\left(-\frac {\sum_{i=1}^m r_i^2}{2l}+O(l^{-1/2})\right)\\
&=l^k\exp\left(-\frac{k^2}{2ml}+O(l^{-1/3})\right)\\
&=l^k\exp\left(-\exp(-2c)/2m+O(l^{-1/6})\right).
\end{split}
\end{align}
Combining the above estimate with \Cref{lem: concentration of St_fixed m}, we have
\begin{align}
\begin{split}
k!\sum_{r_1+\cdots+r_m=k}\prod_{i=1}^m(l)_{r_i}^{-1}\binom{a_{ii}}{r_i}&\ge k!\sum_{\substack{r_1+\cdots+r_m=k\\|r_i-k/m|<l^{1/3},\forall i}}\prod_{i=1}^m(l)_{r_i}^{-1}\binom{a_{ii}}{r_i}\\
&\ge k!\cdot l^{-k}(1+O(l^{-1/6}))\exp\left(\exp(-2c)/2m\right)\binom{\sum_{i=1}^m a_{ii}}{k}\\
%&=(1+O(l^{-1/6}))l^{-k}\exp\left(\exp(-2c)/2m\right)\left(\sum_{i=1}^m a_{ii}\right)_k\\
&\ge(1+O(l^{-1/6}))\exp\left(\exp(-2c)/2m\right)\cdot\prod_{i=0}^{k-1}\left(1+\frac{k_{c_0,c}-i}{l}\right)\\
&=(1+O(l^{-1/6}))\exp(c_0(1-1/m)\exp(-2c))\\
&\ge 1,
\end{split}
\end{align}
which proves \eqref{eq: large fixed point_fixed m}. 

We next deal with \eqref{eq: small fixed point_fixed m}. It suffices to prove that for all $|k-\exp(-c)\sqrt{l}|\le l^{1/3}$ and $\sum_{i=1}^m a_{ii}\le l+k_{c_0,c}$ with $\max_{1\le i\le m}|a_{ii}-l/m|\le l^{2/3}$, one has
$$\mathbf{P}(\diag(\sigma)=(a_{11},\ldots,a_{mm})\mid |S_t|=k)\le\mathbf{P}(\diag(\tau)=(a_{11},\ldots,a_{mm})).$$
Following the explicit expression given in \Cref{lem: distribution of X} and removing the overlapping terms, this is equivalent to
$$k!\sum_{r_1+\cdots+r_m=k}\prod_{i=1}^m(l)_{r_i}^{-1}\binom{a_{ii}}{r_i}\le 1.$$
Observe the order of magnitudes
$$\prod_{i=1}^m(l)_{r_i}^{-1}\le(l)_k^{-1}=O(l^{-k}),\quad k!\sum_{r_1+\cdots+r_m=k}\binom{a_{ii}}{r_i}=\left(\sum_{i=1}^m a_{ii}\right)_k=O(l^k).$$
Therefore, by \Cref{lem: concentration of St_fixed m}, the contribution from the tail part (i.e., $\max_{1\le i\le m}|r_i-k/m|\ge l^{1/3}$) is $O(l^{-1/6})$. Thus, following the estimate in \eqref{eq: estimate of falling factorials_fixed m}, we have
\begin{align}
\begin{split}
k!\sum_{r_1+\cdots+r_m=k}\prod_{i=1}^m(l)_{r_i}^{-1}\binom{a_{ii}}{r_i}&= k!\sum_{\substack{r_1+\cdots+r_m=k\\|r_i-k/m|<l^{1/3},\forall i}}\prod_{i=1}^m(l)_{r_i}^{-1}\binom{a_{ii}}{r_i}+O(l^{-1/6})\\
&\le k!\cdot l^{-k}\exp\left(\exp(-2c)/2m\right)\binom{\sum_{i=1}^m a_{ii}}{k}+O(l^{-1/6})\\
&\le\exp\left(\exp(-2c)/2m\right)\cdot\prod_{i=0}^{k-1}\left(1+\frac{k_{-c_0,c}-i}{l}\right)+O(l^{-1/6})\\
&=(1+O(l^{-1/6}))\exp(-c_0(1-1/m)\exp(-2c))+O(l^{-1/6})\\
&\le 1,
\end{split}
\end{align}
which proves \eqref{eq: small fixed point_fixed m}. 

Based on our preparation \eqref{eq: large fixed point_fixed m} and \eqref{eq: small fixed point_fixed m}, we now decompose $d_{TV}(\diag(\sigma),\diag(\tau))$ as the summation of three regimes:
\begin{align}\label{eq: decompose dTV into three regime_fixed m}
\begin{split}
&d_{TV}(\diag(\sigma),\diag(\tau))\\
&=\frac 12\sum_{a_{11},\ldots,a_{mm}}|\mathbf{P}(\diag(\sigma)=(a_{11},\ldots,a_{mm}))-\mathbf{P}(\diag(\tau)=(a_{11},\ldots,a_{mm}))|\\
&=\frac 12\sum_{\sum_{i=1}^m a_{ii}\ge l+k_{c_0,c}}|\mathbf{P}(\diag(\sigma)=(a_{11},\ldots,a_{mm}))-\mathbf{P}(\diag(\tau)=(a_{11},\ldots,a_{mm}))|\\
&+\frac 12\sum_{\sum_{i=1}^m a_{ii}\le l+k_{-c_0,c}}|\mathbf{P}(\diag(\sigma)=(a_{11},\ldots,a_{mm}))-\mathbf{P}(\diag(\tau)=(a_{11},\ldots,a_{mm}))|\\
&+\frac 12\sum_{l+k_{-c_0,c}<\sum_{i=1}^m a_{ii}<l+k_{c_0,c}}|\mathbf{P}(\diag(\sigma)=(a_{11},\ldots,a_{mm}))-\mathbf{P}(\diag(\tau)=(a_{11},\ldots,a_{mm}))|.
\end{split}
\end{align}
For the regime $\sum_{i=1}^m a_{ii}\ge l+k_{c_0,c}$, combining \eqref{eq: large fixed point_fixed m} with the bound of the probability of the tail event $\max_{1\le i\le m}|X_{ii}-l/m|\ge l^{2/3}$ in \Cref{lem: tail bound and expectation fixed point fixed m}, we have
\begin{align}\label{eq: large fixed point dTV_fixed m}
\begin{split}
&\sum_{\sum_{i=1}^m a_{ii}\ge l+k_{c_0,c}}|\mathbf{P}(\diag(\sigma)=(a_{11},\ldots,a_{mm}))-\mathbf{P}(\diag(\tau)=(a_{11},\ldots,a_{mm}))|\\
&=\mathbf{P}(|F_{n,l}(\sigma)|\ge l+k_{c_0,c})-\mathbf{P}(|F_{n,l}(\tau)|\ge l+k_{c_0,c})+o(1)\\
&=1-\Phi\left(-(1-c_0)\sqrt{1-\frac 1m}\frac{\exp(-c)}{2}\right)-1+\Phi\left((1+c_0)\sqrt{1-\frac 1m}\frac{\exp(-c)}{2}\right)+o(1)\\
&=\Phi\left((1-c_0)\sqrt{1-\frac 1m}\frac{\exp(-c)}{2}\right)+\Phi\left((1+c_0)\sqrt{1-\frac 1m}\frac{\exp(-c)}{2}\right)-1+o(1).
\end{split}
\end{align}
Here, the third line comes from the central limit theorem provided in \Cref{lem: fixed point CLT_fixed m}. For the regime $\sum_{i=1}^m a_{ii}\le l+k_{-c_0,c}$, combining \eqref{eq: small fixed point_fixed m} with the bound of the probability of the tail event $\max_{1\le i\le m}|X_{ii}-l/m|\ge l^{2/3}$ in \Cref{lem: tail bound and expectation fixed point fixed m}, we have
\begin{align}\label{eq: small fixed point dTV_fixed m}
\begin{split}
&\sum_{\sum_{i=1}^m a_{ii}\le l+k_{-c_0,c}}|\mathbf{P}(\diag(\sigma)=(a_{11},\ldots,a_{mm}))-\mathbf{P}(\diag(\tau)=(a_{11},\ldots,a_{mm}))|\\
&=\mathbf{P}(|F_{n,l}(\tau)|\le l+k_{-c_0,c})-\mathbf{P}(|F_{n,l}(\sigma)|\le l+k_{-c_0,c})+o(1)\\
&=\Phi\left((1-c_0)\sqrt{1-\frac 1m}\frac{\exp(-c)}{2}\right)+\Phi\left((1+c_0)\sqrt{1-\frac 1m}\frac{\exp(-c)}{2}\right)-1+o(1).
\end{split}
\end{align}
Here, the third line comes from \Cref{lem: fixed point CLT_fixed m}. For the regime $l+k_{-c_0,c}<\sum_{i=1}^m a_{ii}<l+k_{c_0,c}$, we use the crude bound
\begin{align}\label{eq: middle fixed point dTV_fixed m}
\begin{split}
&\sum_{l+k_{-c_0,c}<\sum_{i=1}^m a_{ii}<l+k_{c_0,c}}|\mathbf{P}(\diag(\sigma)=(a_{11},\ldots,a_{mm}))-\mathbf{P}(\diag(\tau)=(a_{11},\ldots,a_{mm}))|\\
&\le\mathbf{P}(l+k_{-c_0,c}<|F_{n,l}(\sigma)|< l+k_{c_0,c})+\mathbf{P}(l+k_{-c_0,c}<|F_{n,l}(\tau)|< l+k_{c_0,c})\\
&=2\Phi\left((1+c_0)\sqrt{1-\frac 1m}\frac{\exp(-c)}{2}\right)-2\Phi\left((1-c_0)\sqrt{1-\frac 1m}\frac{\exp(-c)}{2}\right)+o(1).
\end{split}
\end{align}
Here, the third line comes from \Cref{lem: fixed point CLT_fixed m}. The estimate \eqref{eq: large fixed point dTV_fixed m}, \eqref{eq: small fixed point dTV_fixed m}, \eqref{eq: middle fixed point dTV_fixed m} and the decomposition \eqref{eq: decompose dTV into three regime_fixed m} together imply
\begin{multline}
\left|d_{TV}(\diag(\sigma),\diag(\tau))-
\left(\Phi\left((1-c_0)\sqrt{1-\frac 1m}\frac{\exp(-c)}{2}\right)+\Phi\left((1+c_0)\sqrt{1-\frac 1m}\frac{\exp(-c)}{2}\right)-1\right)\right|\\
\le2\Phi\left((1+c_0)\sqrt{1-\frac 1m}\frac{\exp(-c)}{2}\right)-2\Phi\left((1-c_0)\sqrt{1-\frac 1m}\frac{\exp(-c)}{2}\right)+o(1).
\end{multline}
Since the above holds for arbitrary $c_0>0$, taking $c_0\rightarrow 0$ completes the proof.
\end{proof}

\subsection{Proof for $m=\omega(1)$}

In this subsection, we let $m=\omega(1)$ goes to infinity. The following lemma gives a tail bound for \(\diag(\cdot)\) under \(U_{\S_n}\) and \(\mu_n^t\). In particular, it shows that \(\sum_{i=1}^m X_{ii}^2\) is, with high probability, much smaller than \(l^2\).

\begin{lemma}\label{lem: tail bound square sum_growing m}
Let \(n=ml\), such that \(m=\omega(1)\) and \(l=\omega(1)\). Suppose
\[
t=\frac n2\left(\log n-\frac12\log l+c\right),
\]
where \(c\in\mathbb R\) is an absolute constant. Let \(\tau\in \S_n\) be distributed
with respect to \(U_{\S_n}\), and let \(\sigma\in \S_n\) be distributed with respect to \(\mu_n^t\).
Then, we have
\[
\mathbf P\left(\sum_{i=1}^m X_{ii}(\tau)^2>l^2\sqrt{\frac 1m+\frac 1l}\right)=O\left(
\sqrt{\frac1m+\frac1l}\right),
\]
and
\[
\mathbf P\left(\sum_{i=1}^m X_{ii}(\sigma)^2>l^2\sqrt{\frac 1m+\frac 1l}\right)=O\left(
\sqrt{\frac1m+\frac1l}\right).
\]
\end{lemma}

\begin{proof}
We first treat \(\tau\). For each \(1\le i\le m\), the random variable \(X_{ii}(\tau)\) has the
hypergeometric distribution corresponding to choosing \(l\) positions out of \(n=ml\), among
which exactly \(l\) belong to \(T_{i,l}\). Hence
\[
\mathbf E[X_{ii}(\tau)]=\frac{l^2}{n}=\frac lm,
\qquad
\mathrm{Var}(X_{ii}(\tau))
=
l\cdot \frac ln\left(1-\frac ln\right)\frac{n-l}{n-1}
\le \frac lm.
\]
Therefore
\[
\mathbf E[X_{ii}(\tau)^2]
=
\mathrm{Var}(X_{ii}(\tau))+\mathbf E[X_{ii}(\tau)]^2
\le
\frac lm+\frac{l^2}{m^2}.
\]
Summing over \(1\le i\le m\), we obtain
\[
\mathbf E\left[\sum_{i=1}^m X_{ii}(\tau)^2\right]
\le
l+\frac{l^2}{m}
=
l^2\left(\frac1m+\frac1l\right).
\]
Now Markov's inequality gives
\[
\mathbf P\left(\sum_{i=1}^m X_{ii}(\tau)^2>l^2\sqrt{\frac1m+\frac1l}\right)
\le
\frac{\mathbf E[\sum_{i=1}^m X_{ii}(\tau)^2]}{l^2\sqrt{\frac1m+\frac1l}}=O\left(
\sqrt{\frac1m+\frac1l}\right).
\]

We next treat \(\sigma\). By the definition of \(\mu_n^t\), we always have
\[
\bigl||S_t|-\exp(-c)\sqrt l\bigr|\le l^{1/3}.
\]
Condition on the set \(S_t\), we write $
R_i:=|S_t\cap T_{i,l}|$ for all $1\le i\le m$. Then
\[
X_{ii}(\sigma)=R_i+Y_i,
\]
where \(Y_i\) is the number of points in \(T_{i,l}\setminus S_t\) that are mapped by the uniform
permutation on \([n]\setminus S_t\) back into \(T_{i,l}\setminus S_t\). Conditional on \(S_t\), the
random variable \(Y_i\) is hypergeometric, so
\[
\mathbf E[Y_i\mid S_t]
=
\frac{(l-R_i)^2}{n-|S_t|},
\qquad
\mathrm{Var}(Y_i\mid S_t)\le \frac{(l-R_i)^2}{n-|S_t|}.
\]
Hence
\begin{align}
\begin{split}
\mathbf E[X_{ii}(\sigma)^2\mid S_t]
&\le
2R_i^2+2\mathbf E[Y_i^2\mid S_t]\\
&=
2R_i^2+2\mathrm{Var}(Y_i\mid S_t)+2\mathbf E[Y_i\mid S_t]^2\\
&\le
2R_i^2+2\frac{(l-R_i)^2}{n-|S_t|}+2\left(\frac{(l-R_i)^2}{n-|S_t|}\right)^2.
\end{split}
\end{align}
Summing over \(1\le i\le m\), we get
\begin{align}
\begin{split}
\mathbf E\left[\sum_{i=1}^m X_{ii}(\sigma)^2\ \middle|\ S_t\right]
&\le
2\sum_{i=1}^m R_i^2+2\sum_{i=1}^m\frac{(l-R_i)^2}{n-|S_t|}+2\sum_{i=1}^m\left(\frac{(l-R_i)^2}{n-|S_t|}\right)^2\\
&\le 2|S_t|^2+2\frac{ml^2}{n/2}+2\frac{ml^4}{(n/2)^2}\\
&=O\left(l^2\left(\frac 1m+\frac 1l\right)\right).
\end{split}
\end{align}
Therefore, we can remove the condition for $S_t$ and directly write $\mathbf E\left[\sum_{i=1}^m X_{ii}(\sigma)^2\right]=O\left(l^2\left(\frac 1m+\frac 1l\right)\right)$. Now, Markov's inequality yields
\[
\mathbf P\left(\sum_{i=1}^m X_{ii}(\sigma)^2>l^2\sqrt{\frac1m+\frac1l}\right)
\le
\frac{\mathbf E[\sum_{i=1}^m X_{ii}(\tau)^2]}{l^2\sqrt{\frac1m+\frac1l}}=O\left(
\sqrt{\frac1m+\frac1l}\right).
\]
This proves the lemma.
\end{proof}

The following lemma is the analogue of \Cref{lem: fixed point CLT_fixed m} in the regime \(m=\omega(1)\).

\begin{lemma}\label{lem: fixed point CLT_growing m}
Let $n=ml$, such that $m=\omega(1)$ and $l=\omega(1)$. Suppose $
t=\frac n2\left(\log n-\frac12\log l+c\right)$, where $c\in \R$ is an absolute constant. Let $\tau\in\S_n$ be distributed with respect to $U_{\S_n}$, and let
$\sigma\in\S_n$ be distributed with respect to $\mu_n^t$. Then, we have
$$
\frac{|F_{n,l}(\tau)|-l}{\sqrt{l}}
\xrightarrow{d}\mathcal N(0,1),\quad\frac{|F_{n,l}(\sigma)|-l-\exp(-c)\sqrt l}{\sqrt{l}}
\xrightarrow{d}\mathcal N(0,1).
$$
\end{lemma} 

\begin{proof}
For the uniform permutation \(\tau\), exactly as before, we have
\[
|F_{n,l}(\tau)|=\sum_{i=1}^n c(i,\tau(i)),
\qquad
c(i,j):=\mathbf 1_{i\equiv j\ (\mathrm{mod}\ m)}.
\]
Hence Hoeffding's theorem applies once we know the variance asymptotics. The same computation
as in the fixed \(m\) case gives
\[
\mathrm{Var}(|F_{n,l}(\tau)|)=\Bigl(1-\frac1m\Bigr)l+o(l)=l+o(l)
\]
since now \(m=\omega(1)\). Thus the same combinatorial central limit theorem yields
\[
\frac{|F_{n,l}(\tau)|-l}{\sqrt l}\xrightarrow{d}\mathcal N(0,1).
\]

For \(\sigma\), we again condition on \(|S_t|=k\) and \((R_1,\ldots,R_m)=(r_1,\ldots,r_m)\). Exactly as in \Cref{lem: fixed point CLT_fixed m}, one writes
\[
|F_{n,l}(\sigma)|=k+Z,
\]
where \(Z\) is a Hoeffding-type permutation statistic on the remaining \(n-k\) points. The same
argument as before gives
\[
\mathbb E|F_{n,l}(\sigma)|
=
l+\Bigl(1-\frac1m\Bigr)\exp(-c)\sqrt l+O(l^{1/3})=l+\exp(-c)\sqrt l+o(l^{1/2}),
\]
which follows \Cref{lem: tail bound and expectation fixed point fixed m} and
\[
\mathrm{Var}(Z\mid k,R_1,\ldots,R_m)=\Bigl(1-\frac1m\Bigr)l+o(l)=l+o(l).
\]
Here, because \(m=\omega(1)\), the factor \(1-\frac1m\) can be replaced by \(1+o(1)\). Hence the
same conditional Hoeffding central limit theorem implies
\[
\frac{|F_{n,l}(\sigma)|-l-\exp(-c)\sqrt l}{\sqrt l}\xrightarrow{d}\mathcal N(0,1).
\]

This proves the lemma.
\end{proof}

\begin{lemma}\label{lem: concentration of St_growing m}
Let \(a_{11},\dots,a_{mm},k\) be fixed such that
\[
\left|\sum_{i=1}^m a_{ii}-l\right|\le l^{2/3},
\qquad
\sum_{i=1}^m a_{ii}^2\le l^2\sqrt{\frac1m+\frac1l},
\qquad
\left|k-\exp(-c)\sqrt l\right|\le l^{1/3}.
\]
Then, we have
\[
\left(1-O\left(\left(\frac1m+\frac1l\right)^{1/4}\right)\right)
\binom{\sum_{i=1}^m a_{ii}}{k}\le\sum_{\substack{r_1+\cdots+r_m=k\\ \sum_{i=1}^m r_i^2\le
l\left(\frac1m+\frac1l\right)^{1/4}}}
\prod_{i=1}^m \binom{a_{ii}}{r_i}
\le\binom{\sum_{i=1}^m a_{ii}}{k}.
\]
\end{lemma}

\begin{proof}
By Vandermonde's identity,
\[
\sum_{r_1+\cdots+r_m=k}\prod_{i=1}^m \binom{a_{ii}}{r_i}
=
\binom{\sum_{i=1}^ma_{ii}}{k}.
\]
Hence the upper bound is immediate. It remains to prove the lower bound. Considering the same experiment as in the proof of \Cref{lem: concentration of St_fixed m}: choose uniformly a
\(k\)-element subset of a set of size \(\sum_{i=1}^m a_{ii}\), which is partitioned into \(m\) disjoint
blocks of sizes \(a_{11},\dots,a_{mm}\). Again, we let \(R_i\) denote the number of chosen elements from the
\(i\)-th block. For each \(1\le i\le m\), since \(R_i\) is hypergeometric, we have
\[
\mathbb E[R_i]=k\frac{a_{ii}}{\sum_{i=1}^m a_{ii}},
\qquad
\mathbb E[R_i(R_i-1)]
=
k(k-1)\frac{a_{ii}(a_{ii}-1)}{\sum_{i=1}^m a_{ii}(\sum_{i=1}^m a_{ii}-1)}\le k(k-1)\frac{a_{ii}^2}{(\sum_{i=1}^m a_{ii})^2}.
\]
Therefore
\begin{align}
\begin{split}
\mathbb E\left[\sum_{i=1}^m R_i^2\right]
&=
\sum_{i=1}^m \mathbb E[R_i(R_i-1)]
+
\sum_{i=1}^m \mathbb E[R_i]\\
&=
k(k-1)\frac{\sum_{i=1}^m a_{ii}^2}{(\sum_{i=1}^m a_{ii})^2}+k\\
&\le
(\exp(-2c)+o(1))\,l\sqrt{\frac1m+\frac1l}
+\exp(-c)\sqrt l+O(l^{1/3}).
\end{split}
\end{align}
Applying Markov's inequality,
\begin{align}
\begin{split}
\mathbf P\left(\sum_{i=1}^m R_i^2>
l\left(\frac1m+\frac1l\right)^{1/4}\right)
&\le
\frac{\mathbb E[\sum_{i=1}^m R_i^2]}
{l\left(\frac1m+\frac1l\right)^{1/4}}\\
&\le (\exp(-2c)+o(1))
\left(\frac1m+\frac1l\right)^{1/4}
+
\frac{\exp(-c)+o(1)}{\sqrt l\,\left(\frac1m+\frac1l\right)^{1/4}}\\
&=O\left(\left(\frac1m+\frac1l\right)^{1/4}\right).
\end{split}
\end{align}
Combining this with the previous Vandermonde's identity, we conclude that
\[
\sum_{\substack{r_1+\cdots+r_m=k\\ \sum_{i=1}^m r_i^2\le
l\left(\frac1m+\frac1l\right)^{1/4}}}
\prod_{i=1}^m \binom{a_{ii}}{r_i}
\ge
\left(1-O\left(\left(\frac1m+\frac1l\right)^{1/4}\right)\right)
\binom{\sum_{i=1}^m a_{ii}}{k},
\]
which proves the lemma.
\end{proof}

\begin{proof}[Proof of \Cref{thm: restate of limiting profile}, $m=\omega(1)$ case] 
Let $\tau\in\S_n$ be distributed with respect to $U_{\S_n}$, and let
$\sigma\in\S_n$ be distributed with respect to $\mu_n^t$. By \Cref{prop: transfer to diag}, it suffices to prove
$$d_{TV}(\diag(\sigma),\diag(\tau))=2\Phi\left(\frac{\exp(-c)}{2}\right)-1+o(1).$$
Let $c_0>0$ be another absolute constant, and write
$$k_{c_0,c}:=(1/2+c_0)\exp(-c)\sqrt l,\quad k_{-c_0,c}:=(1/2-c_0)\exp(-c)\sqrt l.$$
We claim that the following holds when $l$ is sufficiently large. For all sequence $(a_{11},\ldots,a_{mm})$ such that
\begin{equation}\label{eq: concentrated regime_growing m}
\left|\sum_{i=1}^m a_{ii}-l\right|\le l^{2/3},
\qquad
\sum_{i=1}^m a_{ii}^2\le l^2\sqrt{\frac1m+\frac1l},
\end{equation}
one has 
\begin{equation}\label{eq: large fixed point_growing m}
\sum_{i=1}^m a_{ii}\ge l+k_{c_0,c}\text{ implies }\mathbf{P}(\diag(\sigma)=(a_{11},\ldots,a_{mm}))\ge \mathbf{P}(\diag(\tau)=(a_{11},\ldots,a_{mm})),
\end{equation}
\begin{equation}\label{eq: small fixed point_growing m}
\sum_{i=1}^m a_{ii}\le l+k_{-c_0,c}\text{ implies }\mathbf{P}(\diag(\sigma)=(a_{11},\ldots,a_{mm}))\le\mathbf{P}(\diag(\tau)=(a_{11},\ldots,a_{mm})).
\end{equation}
We first deal with \eqref{eq: large fixed point_growing m}. It suffices to prove that for all $|k-\exp(-c)\sqrt{l}|\le l^{1/3}$ and $\sum_{i=1}^m a_{ii}\ge l+k_{c_0,c}$ satisfying \eqref{eq: concentrated regime_growing m}, one has
$$\mathbf{P}(\diag(\sigma)=(a_{11},\ldots,a_{mm})\mid |S_t|=k)\ge\mathbf{P}(\diag(\tau)=(a_{11},\ldots,a_{mm})).$$
Following the explicit expression given in \Cref{lem: distribution of X} and removing the overlapping terms, this is equivalent to
$$k!\sum_{r_1+\cdots+r_m=k}\prod_{i=1}^m(l)_{r_i}^{-1}\binom{a_{ii}}{r_i}\ge1.$$
Here, $(l)_{r_i}:=l(l-1)\cdots(l-r_i+1)$ denotes the falling factorial. In particular, when $\sum_{i=1}^m r_i^2\le
l\left(\frac1m+\frac1l\right)^{1/4}$, we have
\begin{align}\label{eq: estimate of falling factorials_growing m}
\begin{split}
\prod_{i=1}^m (l)_{r_i}&=l^k\prod_{i=1}^m (1-1/l)\cdots(1-(r_i-1)/l)\\
&=l^k\exp\left(-O\left(\sum_{i=1}^m\left(1/l+\cdots+(r_i-1)/l\right)\right)\right)\\
&=l^k\exp\left(-O\left(\sum_{i=1}^m r_i^2/l\right)\right)\\
&=l^k(1-o(1)).
\end{split}
\end{align}
Combining the above estimate with \Cref{lem: concentration of St_growing m}, we have
\begin{align}
\begin{split}
k!\sum_{r_1+\cdots+r_m=k}\prod_{i=1}^m(l)_{r_i}^{-1}\binom{a_{ii}}{r_i}&\ge k!\sum_{\substack{r_1+\cdots+r_m=k\\ \sum_{i=1}^m r_i^2\le
l\left(\frac1m+\frac1l\right)^{1/4}}}\prod_{i=1}^m(l)_{r_i}^{-1}\binom{a_{ii}}{r_i}\\
&\ge k!\cdot l^{-k}(1-o(1))\exp\binom{\sum_{i=1}^m a_{ii}}{k}\\
&\ge(1-o(1))\cdot\prod_{i=0}^{k-1}\left(1+\frac{k_{c_0,c}-i}{l}\right)\\
&=(1-o(1))\exp(c_0\exp(-2c))\\
&\ge 1,
\end{split}
\end{align}
which proves \eqref{eq: large fixed point_growing m}. 

We next deal with \eqref{eq: small fixed point_growing m}. It suffices to prove that for all $|k-\exp(-c)\sqrt{l}|\le l^{1/3}$ and $\sum_{i=1}^m a_{ii}\le l+k_{c_0,c}$ satisfying \eqref{eq: concentrated regime_growing m}, one has
$$\mathbf{P}(\diag(\sigma)=(a_{11},\ldots,a_{mm})\mid |S_t|=k)\le\mathbf{P}(\diag(\tau)=(a_{11},\ldots,a_{mm})).$$
Following the explicit expression given in \Cref{lem: distribution of X} and removing the overlapping terms, this is equivalent to
$$k!\sum_{r_1+\cdots+r_m=k}\prod_{i=1}^m(l)_{r_i}^{-1}\binom{a_{ii}}{r_i}\le 1.$$
Observe the order of magnitudes
$$\prod_{i=1}^m(l)_{r_i}^{-1}\le(l)_k^{-1}=O(l^{-k}),\quad k!\sum_{r_1+\cdots+r_m=k}\binom{a_{ii}}{r_i}=\left(\sum_{i=1}^m a_{ii}\right)_k=O(l^k).$$ 
Therefore, by \Cref{lem: concentration of St_growing m}, the contribution from the tail part (i.e., $\sum_{i=1}^m r_i^2>
l\left(\frac1m+\frac1l\right)^{1/4}$) is $O\left(\left(\frac1m+\frac1l\right)^{1/4}\right)$. Thus, following the estimate in \eqref{eq: estimate of falling factorials_growing m}, we have
\begin{align}
\begin{split}
k!\sum_{r_1+\cdots+r_m=k}\prod_{i=1}^m(l)_{r_i}^{-1}\binom{a_{ii}}{r_i}&= k!\sum_{\substack{r_1+\cdots+r_m=k\\\sum_{i=1}^m r_i^2\le
l\left(\frac1m+\frac1l\right)^{1/4}}}\prod_{i=1}^m(l)_{r_i}^{-1}\binom{a_{ii}}{r_i}+O\left(\left(\frac1m+\frac1l\right)^{1/4}\right)\\
&\le k!\cdot l^{-k}\binom{\sum_{i=1}^m a_{ii}}{k}+O\left(\left(\frac1m+\frac1l\right)^{1/4}\right)\\
&\le\prod_{i=0}^{k-1}\left(1+\frac{k_{-c_0,c}-i}{l}\right)+O\left(\left(\frac1m+\frac1l\right)^{1/4}\right)\\
&=(1+o(1))\exp(-c_0\exp(-2c))+O\left(\left(\frac1m+\frac1l\right)^{1/4}\right)\\
&\le 1,
\end{split}
\end{align}
which proves \eqref{eq: small fixed point_growing m}. 

Based on our preparation \eqref{eq: large fixed point_growing m}, \eqref{eq: small fixed point_growing m} and the fact that
$$\mathbf{P}(|F_{n,l}(\tau)-l|>l^{2/3})=o(1),\quad\mathbf{P}(|F_{n,l}(\sigma)-l|>l^{2/3})=o(1),$$
which is ensured by \Cref{lem: fixed point CLT_growing m}, we now decompose $d_{TV}(\diag(\sigma),\diag(\tau))$ as the following:
\begin{align}\label{eq: decompose dTV into three regime_growing m}
\begin{split}
&d_{TV}(\diag(\sigma),\diag(\tau))\\
&=\frac 12\sum_{l+k_{c_0,c}\le\sum_{i=1}^m a_{ii}\le l+l^{2/3}}\mathbf{P}(\diag(\sigma)=(a_{11},\ldots,a_{mm}))-\mathbf{P}(\diag(\tau)=(a_{11},\ldots,a_{mm}))|\\
&+\frac 12\sum_{l-l^{2/3}\le\sum_{i=1}^m a_{ii}\le l+k_{-c_0,c}}|\mathbf{P}(\diag(\sigma)=(a_{11},\ldots,a_{mm}))-\mathbf{P}(\diag(\tau)=(a_{11},\ldots,a_{mm}))|\\
&+\frac 12\sum_{l+k_{-c_0,c}<\sum_{i=1}^m a_{ii}<l+k_{c_0,c}}|\mathbf{P}(\diag(\sigma)=(a_{11},\ldots,a_{mm}))-\mathbf{P}(\diag(\tau)=(a_{11},\ldots,a_{mm}))|\\
&+o(1).
\end{split}
\end{align}
For the regime $l+k_{c_0,c}\le\sum_{i=1}^m a_{ii}\le l+l^{2/3}$, combining \eqref{eq: large fixed point_growing m} with the bound of the probability of the tail event $\sum_{i=1}^m X_{ii}^2>l^2\sqrt{\frac 1m+\frac 1l}$ in \Cref{lem: tail bound square sum_growing m}, we have
\begin{align}\label{eq: large fixed point dTV_growing m}
\begin{split}
&\sum_{l+k_{c_0,c}\le\sum_{i=1}^m a_{ii}\le l+l^{2/3}}|\mathbf{P}(\diag(\sigma)=(a_{11},\ldots,a_{mm}))-\mathbf{P}(\diag(\tau)=(a_{11},\ldots,a_{mm}))|\\
&=\mathbf{P}(l+k_{c_0,c}\le|F_{n,l}(\sigma)|\le l+l^{2/3})-\mathbf{P}(l+k_{c_0,c}\le|F_{n,l}(\tau)|\le l+l^{2/3})+o(1)\\
&=1-\Phi\left(-(1-c_0)\frac{\exp(-c)}{2}\right)-1+\Phi\left((1+c_0)\frac{\exp(-c)}{2}\right)+o(1)\\
&=\Phi\left((1-c_0)\frac{\exp(-c)}{2}\right)+\Phi\left((1+c_0)\frac{\exp(-c)}{2}\right)-1+o(1).
\end{split}
\end{align}
Here, the third line comes from the central limit theorem provided in \Cref{lem: fixed point CLT_growing m}. For the regime $l-l^{2/3}\le\sum_{i=1}^m a_{ii}\le l+k_{-c_0,c}$, combining \eqref{eq: small fixed point_growing m} with the bound of the probability of the tail event $\sum_{i=1}^m X_{ii}^2>l^2\sqrt{\frac 1m+\frac 1l}$ in \Cref{lem: tail bound square sum_growing m}, we have
\begin{align}\label{eq: small fixed point dTV_growing m}
\begin{split}
&\sum_{l-l^{2/3}\le\sum_{i=1}^m a_{ii}\le l+k_{-c_0,c}}|\mathbf{P}(\diag(\sigma)=(a_{11},\ldots,a_{mm}))-\mathbf{P}(\diag(\tau)=(a_{11},\ldots,a_{mm}))|\\
&=\mathbf{P}(l-l^{2/3}\le|F_{n,l}(\tau)|\le l+k_{-c_0,c})-\mathbf{P}(l-l^{2/3}\le|F_{n,l}(\sigma)|\le l+k_{-c_0,c})+o(1)\\
&=\Phi\left((1-c_0)\frac{\exp(-c)}{2}\right)+\Phi\left((1+c_0)\frac{\exp(-c)}{2}\right)-1+o(1).
\end{split}
\end{align}
Here, the third line comes from \Cref{lem: fixed point CLT_growing m}. For the regime $l+k_{-c_0,c}<\sum_{i=1}^m a_{ii}<l+k_{c_0,c}$, we use the crude bound
\begin{align}\label{eq: middle fixed point dTV_growing m}
\begin{split}
&\sum_{l+k_{-c_0,c}<\sum_{i=1}^m a_{ii}<l+k_{c_0,c}}|\mathbf{P}(\diag(\sigma)=(a_{11},\ldots,a_{mm}))-\mathbf{P}(\diag(\tau)=(a_{11},\ldots,a_{mm}))|\\
&\le\mathbf{P}(l+k_{-c_0,c}<|F_{n,l}(\sigma)|< l+k_{c_0,c})+\mathbf{P}(l+k_{-c_0,c}<|F_{n,l}(\tau)|< l+k_{c_0,c})\\
&=2\Phi\left((1+c_0)\frac{\exp(-c)}{2}\right)-2\Phi\left((1-c_0)\frac{\exp(-c)}{2}\right)+o(1).
\end{split}
\end{align}
Here, the third line comes from \Cref{lem: fixed point CLT_growing m}. The estimate \eqref{eq: large fixed point dTV_growing m}, \eqref{eq: small fixed point dTV_growing m}, \eqref{eq: middle fixed point dTV_growing m} and the decomposition \eqref{eq: decompose dTV into three regime_growing m} together imply
\begin{multline}
\left|d_{TV}(\diag(\sigma),\diag(\tau))-
\left(\Phi\left((1-c_0)\frac{\exp(-c)}{2}\right)+\Phi\left((1+c_0)\frac{\exp(-c)}{2}\right)-1\right)\right|\\
\le2\Phi\left((1+c_0)\frac{\exp(-c)}{2}\right)-2\Phi\left((1-c_0)\frac{\exp(-c)}{2}\right)+o(1).
\end{multline}
Since the above holds for arbitrary $c_0>0$, taking $c_0\rightarrow 0$ completes the proof.
\end{proof}

\begin{rmk}
One might wonder whether the total variation distance between the laws of \(\diag(\sigma)\) and \(\diag(\tau)\) is always equal to the total variation distance between the laws of \(|F_{n,l}(\sigma)|\) and \(|F_{n,l}(\tau)|\). This is, however, not true in general. Here, \(\tau\) is sampled uniformly from \(\mathfrak S_n\), whereas \(\sigma\) is sampled by first choosing a random subset of points to be fixed, and then taking a uniformly random permutation on the remaining points.

Indeed, consider the case \(m=l=3\), so that $n=9$. Let \(\tau\) be uniformly distributed over \(\mathfrak S_9\), and let \(\sigma\) be sampled by first choosing two points to fix and then taking a uniformly random permutation on the remaining seven points. Following direct calculation, one can obtain
\[
d_{TV}(\diag(\sigma),\diag(\tau))=\frac{103}{252}\approx 0.409,
\]
whereas
\[
d_{TV}(|F_{9,3}(\sigma)|,|F_{9,3}(\tau)|)=\frac{337}{840}\approx 0.401.
\]
Therefore, the two total variation distances do not coincide.
\end{rmk}

\bibliographystyle{plain}
\bibliography{references.bib}

\end{document}